\documentclass[a4paper]{amsart}

\usepackage[backend=biber, isbn=false, doi=false, url=false, style=numeric, sortcites=true, citestyle=numeric, giveninits=true]{biblatex}
\bibliography{bibliografia.bib}

\usepackage{xfrac}
\usepackage{amssymb}
\usepackage{amsmath}
\usepackage{amsthm}
\usepackage{csquotes}
\usepackage{stmaryrd}
\usepackage{wrapfig}
\usepackage{xspace}
\usepackage{tikz}
\usetikzlibrary{cd}
\usetikzlibrary{calc}
\usepackage{adjustbox}
\usepackage{capt-of}
\tikzstyle{dot}=[inner sep=1pt, fill, black, circle, draw, minimum size = 5pt]
\tikzstyle{highlight}=[inner sep=1pt, draw=black, fill=white, circle, minimum size = 9pt]
\usepackage[capitalise, noabbrev]{cleveref}
\usepackage{enumitem}
\setlist[itemize]{noitemsep}
\setlist[enumerate]{noitemsep}
\setlist[enumerate,1]{label = (\roman*)}

\newcommand{\eqrel}{ {\sim} }
\newcommand{\neqrel}{ {\nsim} }

\newtheorem{theorem}{Theorem}
\newtheorem{proposition}[theorem]{Proposition}

\theoremstyle{definition}
\newtheorem{definition}[theorem]{Definition}
\newtheorem{lemma}[theorem]{Lemma}
\newtheorem{corollary}[theorem]{Corollary}

\newcommand\dna{\mathtt{DNA}}
\newcommand\dnal[1]{\mathtt{#1}}

\newcommand\cpc{\mathtt{CPC}}
\newcommand\ipc{\mathtt{IPC}}

\newcommand\inqB{\mathtt{InqB}}

\newcommand\vari{\mathcal{V}}

\newcommand\dvari{\mathbb{D}}

\newcommand\class{\mathcal{C}}

\renewcommand{\restriction}{ {\upharpoonright} }

\newcommand{\dom}{\mathsf{dom}}
\newcommand{\depth}{\mathsf{depth}}
\newcommand{\width}{\mathsf{width}}
\newcommand{\height}{\mathsf{depth}}

\newcommand{\rank}{\mathsf{rank}}

\newcommand\at{\mathtt{AT}}
\newcommand{\espace}[1]{\mathfrak{#1}}
\newcommand{\clop}{\mathcal{C}}
\newcommand{\upset}{\mathcal{U}}

\newcommand{\Esa}{\mathsf{Esa}}
\newcommand{\sem}[1]{\llbracket #1 \rrbracket}
\newcommand{\Log}{Log}
\newcommand{\Int}{\mathrm{Int}}
\newcommand{\Cl}{\mathrm{Cl}}
\newcommand{\PF}{\mathcal{PF}}
\newcommand{\CU}{\mathcal{CU}}

\newcommand{\comp}[1]{\overline{#1}}

\newcommand{\cA}{\mathcal{A}}

\newcommand{\cC}{\mathcal{C}}

\newcommand{\cE}{\mathcal{E}}

\newcommand{\cX}{\mathcal{X}}

\newcommand{\fE}{\mathfrak{E}}

\newcommand{\fF}{\mathfrak{F}}
\newcommand{\fG}{\mathfrak{G}}

\newcommand{\fR}{\mathfrak{R}}
\newcommand{\fS}{\mathfrak{S}}

\newcommand{\BA}{\mathsf{BA}}
\newcommand{\HA}{\mathsf{HA}}
\newcommand{\HArfsi}{\mathsf{HA_{RFSI}}}

\newcommand\dep{\mathop{=\!}\xspace}

\newcommand\langInqI{\mathcal{L}_{\ipc}^\otimes}
\newcommand\langInt{\mathcal{L}_{\ipc}}

\title{Esakia Duals of Regular Heyting Algebras}
\date{\today}
\author{Gianluca Grilletti}
\author{Davide Emilio Quadrellaro}
\email{grilletti.gianluca@gmail.com}
\email{davide.quadrellaro@helsinki.fi}
\address{Munich Centre for Mathematical Philosophy,	Munich, Germany}
\address{Department of Mathematics and Statistics, University of Helsinki, P.O. Box 68 (Pietari Kalmin katu 5), 00014 Helsinki, Finland}
\thanks{The authors would like to thank Nick Bezhanishvili for suggesting the main question studied in this work. We are also grateful to Fan Yang for many helpful comments on the subject of this paper and to an anonymous reviewer for their very useful suggestions and observations.
}
\keywords{Heyting algebras, Esakia spaces, duality theory}
\subjclass[2020]{06D20, 03C05, 03B55}

\begin{document}
	
		\begin{abstract}
	   We investigate in this article regular Heyting algebras by means of Esakia duality. In particular, we give a characterisation of Esakia spaces  dual to regular Heyting algebras and we show that there are continuum-many varieties of Heyting algebras generated by regular Heyting algebras. We also study several logical applications of these classes of objects and we use them to provide novel topological completeness theorems for inquisitive logic, $\dna$-logics and dependence logic.
	\end{abstract}
\maketitle

\section*{Introduction}

A Heyting algebra is said to be \textit{regular} if it is generated by its subset of regular elements, i.e. elements $x$ which are identical to their double negation $\neg \neg x$. In this article we investigate regular Heyting algebras from the viewpoint of Esakia duality and we establish several connections to intermediate logics, inquisitive logic, dependence logic and  $\dna$-logic \cite{Ciardelli2011-CIAIL,Yang2016-YANPLO,bezhanishvili_grilletti_quadrellaro_2022}. 

Regular Heyting algebras have recently come to attention for their role in the algebraic semantics of inquisitive logic \cite{grilletti} and more generally of so-called $\dna$-logics \cite{Quadrellaro.2019,bezhanishvili_grilletti_quadrellaro_2022}. $\dna$-logics, for \textit{double negation on atoms}, make for an interesting generalisation of inquisitive logic which arises when considering translations of intermediate logics under the double negation map. More precisely, $\dna$-logics are those set of formulas $L^\neg$ which contain a formula $\phi$ whenever the intermediate logic $L$ contains $\phi[\sfrac{\overline{\neg p}}{\overline{p}}]$, namely  the formula obtained by replacing simultaneously $p$ by $\neg p$ for any variable $p$. Such logics were originally introduced in \cite{Miglioli1989-PIESRO} and it was shown already in \cite{Ciardelli.2009} that inquisitive logic is a paradigmatic example of them.

Regular Heyting algebras also play a role in dependence logic. The connection between the team semantics of dependence logic and Heyting algebras was originally pointed out in \cite[\S 3]{Abramsky2009-ABRFIT} and it was later proved in \cite{quadrellaro2021intermediate} that suitable expansions of regular Heyting algebras provide an algebraic semantics to propositional dependence logic. It was shown in \cite{nakov.quadrellaro} that such algebraic semantics, both for inquisitive, dependence and $\dna$-logics, are unique in the sense provided by a suitable notion of algebraizability for non-standard logics.

In this article we supplement the previous work on the subject by investigating regular Heyting algebras from the perspective of duality theory. Firstly, in \cref{preliminaries}, we review Esakia duality and the previous results on inquisitive and $\dna$-logics. In \cref{subsection:stoneSpaceMaximalElements} we provide detailed proofs for several folklore results on the relation between regular clopen upsets and the Stone subspaces of maximal elements of an Esakia space.

In \cref{sec.characterisation} we consider at length the main question of this article and we provide two characterisations of (finite)  Esakia spaces dual to (finite) regular Heyting algebras.  In \cref{morphism.characterisation} we give a first characterisation of finite regular Esakia spaces in terms of p-morphisms, while in \cref{quotient.characterisation} we provide a necessary condition for an arbitrary Esakia space to be regular based on suitable equivalence relations, and we also give an alternative description of finite regular Esakia spaces. These characterisations allow us to consider a problem originally posed to us by Nick Bezhanishvili in a personal communication: \textit{how many varieties of Heyting algebras are generated by regular Heyting algebras?} In \cref{cardinality.sublattice} we answer this question by showing that there are continuum-many of such varieties, complementing the result from \cite{bezhanishvili_grilletti_quadrellaro_2022} showing that the sublattice of  regularly generated varieties extending $\mathtt{ML}$ is dually isomorphic to $\omega+1$. 

Finally, in \cref{applications.logic}, we apply the previous results to the context of $\dna$-logics, inquisitive and dependence logic and we provide a  topological semantics to these logical systems. We conclude the paper in \cref{conclusion} by highlighting some possible ideas of further research.

\section{Preliminaries}\label{preliminaries}

We recall in this section the preliminary notions needed later in the paper.  We review the algebraic semantics of intermediate and $\dna$-logics, the Esakia duality between Heyting algebras and Esakia spaces, and fix some notational conventions used throughout the paper. We refer the reader to \cite{Burris.1981,9783030120955, Zakharyaschev.1997, Johnstone} for a detailed presentation of these notions and results.

\subsection{Orders, Lattices, Heyting Algebras}
For $(P,\leq)$ a partial order and $Q\subseteq P$ we indicate with $Q^\uparrow$ and $Q^\downarrow$ the upset and downset generated by $Q$ respectively, that is
\begin{equation*}
Q^\uparrow = \{ p\in P  \mid  \exists q \in Q. q \leq p \}
\qquad
Q^\downarrow = \{ p\in P  \mid  \exists q \in Q.q \geq p  \}.
\end{equation*}
For $p\in P$, we write $p^\uparrow$ and $p^\downarrow$ for the sets $\{p\}^\uparrow$ and $\{p\}^\downarrow$ respectively.
We call a set $Q$ such that $Q = Q^{\uparrow}$ an \emph{upset}, and similarly we call a set $R$ such that $R = R^{\downarrow}$ a \emph{downset}. Given a finite poset $P$, we define the depth $\depth(p)$ of an element $p\in P$ as the size of a maximal chain in $p^\uparrow\setminus\{p\}$. We define $\height(P):=\text{sup}\{\depth(p)+1 \mid p\in P  \}$ and $\width(P)$ as the size of the greatest antichain in $P$.	

A \emph{Heyting algebra} is a structure $(H,\land,\vee,\to,1,0)$ where $(H,\land,\vee,1,0)$ is a bounded distributive lattice and  $\to$ is a binary operation on $H$ such that for every $a,b,c\in H$ we have $a \leq b\to c$ if and only if $a \land b \leq c$. Henceforth, we will write $H$ to indicate a Heyting algebra (i.e., omitting the signature) for brevity.
We use the symbol $\HA$ to indicate the class of all Heyting algebras. The class $\HA$ of Heyting algebras is \emph{equationally defined}, that is, a \emph{variety}. With a slight abuse of notation, we also use the notation $\HA$ to indicate the \emph{category of Heyting algebras}, whose objects are Heyting algebras and whose arrows are algebra homomorphisms.

A \emph{Boolean algebra} $B$ is a Heyting algebra satisfying the equation $x=\neg \neg x$ for all $x\in B$. We write $\BA$ for both the class and the category of Boolean algebras. For any Heyting algebra $H$, we say that $x\in H$ is \emph{regular} if $x=\neg \neg x$ and we let $H_\neg:=\{x\in H \mid x=\neg \neg x \}$. One can verify that $H_\neg$ is a subalgebra of $H$ with respect to its $\{\land,\to,0,1  \}$-reduct and that it forms a Boolean algebra with join $x\dot{\lor}y:= \neg (\neg x\land \neg y) $. We say that a Heyting algebra is \textit{regular}, or \textit{regularly generated}, if $H=\langle H_\neg \rangle$, where $ \langle H_\neg \rangle $ refers to the subalgebra of $H$ generated by $H_\neg$. 

We also recall that varieties are exactly those classes of algebras which are closed under subalgebras $\mathbb{S}$, products $\mathbb{P}$ and homomorphic images $\mathbb{H}$. We write $\mathbb{V}(\cC)$ for the smallest variety containing a class of algebras $\cC$.

\subsection{Esakia Duality}\label{subsection:esakia duality}

We recall the Esakia duality between Heyting algebras and  Esakia spaces. We refer the reader to \cite{9783030120955} for more details on Esakia spaces and Esakia duality.

Given a topological space $(X,\tau)$ we write $\mathcal{C}(X)$ for its collection of clopen subsets, i.e. subsets $U\subseteq X$ which are both open and closed in the $\tau$-topology. For ease of read, in the remainder of the paper we will omit the reference to the topology $\tau$ and simply write $X$ to indicate a topological space. If such notation is needed for a given space $X$, we then  write $\tau_X$ for the collection of its open sets.

Recall that a topological space is \emph{totally disconnected} if its only connected components are singletons. A \emph{Stone space} is a compact, Hausdorff and totally disconnected space. Stone duality states that the category of Stone spaces with continuous maps is dually equivalent to the category of Boolean algebras with homomorphisms.

Esakia duality provides an analogue of this result for Heyting algebras. We define Esakia spaces as follows. 

\begin{definition}[Esakia Space]
	Let $\mathfrak{E}=(X,\leq)$ consist of a topological space $X$ and a partial order $\leq$ over $X$. We say that $\mathfrak{E}$ is an \textit{Esakia Space} if:
	\begin{itemize}
		\item[(i)]  $X$ is a compact space;
		\item [(ii)] For all $x,y\in \espace{E}$ such that $x\nleq y$, there is a clopen upset $U$ such that $x\in U$ and $y\notin U$;
		\item[(iii)] If $U$ is a clopen set, then also $U^\downarrow$ is clopen.
	\end{itemize}
\end{definition}

\noindent Condition (ii) in the definition above is called \textit{Priestley Separation Axiom}. Spaces satisfying conditions (i) and (ii) are called \textit{Priestley Spaces} \cite[\S 11]{pries}, hence every Esakia Space is also a Priestley space. Moreover, it can be also verified that every Esakia space is a Stone space. We write $\mathcal{CU}(\mathfrak{E})$ for the set of clopen upsets over $\mathfrak{E}$.

We write $\Esa$ to indicate the class of Esakia spaces. In analogy with $\HA$, we can see $\Esa$ as a category whose objects are Esakia spaces.	A morphism between Esakia spaces is a map that preserves the topological structure, the order-theoretical structure and the relation between the two. 

\begin{definition}[p-morphism]
	Given Esakia spaces $\espace{E} = (X, \leq)$ and $\espace{E'} = {(X',\leq)}$, a \emph{p-morphism} $f: \espace{E} \to \espace{E'}$ is a continuous map such that:
	\begin{enumerate}
		\item For all $x,y \in \espace{E}$, if $x \leq y$ then $f(x) \leq f(y)$;
		\item For all $x\in \espace{E}$ and $y' \in \espace{E'}$ such that $f(x) \leq y'$, there exists $y \in \espace{E}$ such that $x \leq y$ and $f(y) = y'$.
	\end{enumerate}
\end{definition}

\noindent	The continuity of the map ensures that the preimage of a clopen in $\clop(\espace{E'})$ is contained in $\clop(\espace{E})$.	Additionally, condition (i) ensures that the preimage of upsets (downsets) of $\espace{E'}$ are again upsets (downsets) of $\espace{E}$. We write $f:\fE\twoheadrightarrow \fE'$ when $f$ is a surjective p-morphism from $\fE$ to $\fE'$.

Esakia spaces allow us to provide a duality for Heyting algebras, in the same spirit of the Stone duality for Boolean algebras or the Priestley duality for bounded distributive lattices.   Since they will play a major role in the rest of the paper, we recall what are the underlying functors of this duality.

Given a Heyting algebra $H$, a proper subset $F\subsetneq H$ is a \textit{prime filter} if it is a filter and, whenever $x\lor y\in F$, then $x\in F$ or $y\in F$. Let $X_H=\PF(H)$ be the set of all prime filters over $H$, we can endow $X_H$ with a topology $\tau_H$, having as subbasis the following family of sets:
\[ \{ \phi(a) \,|\, a\in H \} \cup \{ \phi(a)^c \,| \, a\in H \} \]
where  $\phi(a) = \{ F\in X_H \,|\, a\in F \}$ and where $\phi(a)^c$ denotes the complement of $\phi(a)$ in $X_H$. Moreover, if we consider the standard inclusion order $\subseteq$ between prime filters, the ordered space $\espace{E}_H = (X_H, \tau_H, \subseteq)$ so obtained is an Esakia space: we call this the \emph{Esakia dual of $H$}. 

On the other hand, if $\fE$ is an Esakia Space we can define the Heyting algebra $H_{\espace{E}}$ over the set $\CU(\fE)$ of clopen upsets of $\fE$:
\begin{equation*}
\begin{array}{r@{\hspace{.3em}}c@{\hspace{.3em}}l  @{\hspace{1.5em}}  r@{\hspace{.3em}}c@{\hspace{.3em}}l  @{\hspace{1.5em}}  r@{\hspace{.3em}}c@{\hspace{.3em}}l}
U \land V &= & U \cap V
&U \lor V &= & U \cup V
&U \rightarrow V &= & ((U\setminus V)^\downarrow)^c
\end{array}
\end{equation*}	
\noindent
where $U^c$ denotes the complement of $U$ in $\espace{E}$.  We shall also write $\overline{U}$ for  $(U^{\downarrow})^c$, namely for the pseudocomplement of $U$ in $\CU(\fE)$. The algebra $H_{\espace{E}}$ is a Heyting algebra, which we call the \emph{Esakia dual of $\espace{E}$}. Esakia proved that these two maps are functorial and describe a dual equivalence between $\HA$ and $\Esa$, in particular the following holds with respect to objects.  

\begin{theorem}[Esakia]\label{theorem:esakiaDuality}
	For every Heyting algebra $H$, we have $H\cong H_{\mathfrak{E}_H}$. For every Esakia Space $\mathfrak{E}$, we have $\mathfrak{E}\cong \mathfrak{E}_{H_{\mathfrak{E}}}$.
\end{theorem}

\noindent 	At the level of arrows, we have the following correspondence:
\begin{itemize}
	\item Given a homomorphism $f: H \to H'$ between two Heyting algebras, we define the p-morphism $\hat{f}: \espace{E}_{H'} \to \espace{E}_{H}$ by $\hat{f}(x) = f^{-1}[x]$; 
	\item Given a p-morphism $g: \espace{E} \to \espace{E'}$ between two Esakia spaces, we define the homomorphism $\hat{g}: H_{\espace{E'}} \to H_{\espace{E}}$ by $\hat{g}(U) = g^{-1}[U]$.
\end{itemize}

\noindent These mappings, together with the ones presented above, provide a full duality between the categories $\HA$ and $\Esa$.  We indicate with $\CU: \HA \to \Esa$ and $\PF: \Esa \to \HA$ the corresponding functors. 

When restricted to the finite setting, Esakia duality delivers a dual equivalence between finite Heyting algebras and finite Esakia spaces. Since an Esakia space $\espace{E}$ is a Stone space, in the finite case its topology is discrete. This allows to study finite Esakia spaces only in terms of their order-theoretic structure and to treat them simply as finite partial orders.

\subsection{Semantics for Intermediate Logics}
\label{subsection:semanticsIntermediate}

Heyting algebras and Esakia Spaces are closely connected to intermediate logics, namely those logics which lie between intuitionistic and classical propositional logic. Let $\at$ be a set of atomic variables and consider the set of formulas $\langInt$ generated by the following grammar:
\begin{align*}
\phi ::= p \mid  \bot \mid \top \mid\phi\land\phi \mid \phi\lor \phi \mid  \phi \rightarrow \phi
\end{align*}

\noindent where $p\in\at$. We write $\ipc$ for intuitionistic logic and $\cpc$ for classical propositional logic. There is a standard way to interpret these formulas on Heyting algebras -- see e.g.  \cite[Sec. 7.3]{Zakharyaschev.1997}.  Given a Heyting algebra $H$ and a map $\mu: \at \to H$ (also called a \emph{valuation}), we can interpret formulas of $\langInt$ on $H$ inductively as follows:
\begin{equation*}
\begin{array}{r@{\hspace{.1em}}l @{\hspace{1em}}  r@{\hspace{.1em}}l @{\hspace{1em}}  r@{\hspace{.1em}}c@{\hspace{.1em}}l}	
\llbracket p \rrbracket^{H,\mu} &= \mu(p)
&\llbracket \bot \rrbracket^{H,\mu} &= 0 \\[0.5em]
\llbracket \top \rrbracket^{H,\mu} &=  1  
&\llbracket \phi \land \psi \rrbracket^{H,\mu} &= \llbracket \phi \rrbracket^{H,\mu} \land \llbracket \psi \rrbracket^{H,\mu} \\[0.5em]
\llbracket \phi \rightarrow \psi \rrbracket^{H,\mu} &= \llbracket \phi \rrbracket^{H,\mu} \to \llbracket \psi \rrbracket^{H,\mu}
&\llbracket \phi \lor \psi \rrbracket^{H,\mu} &= \llbracket\phi \rrbracket^{H,\mu} \vee  \llbracket\psi \rrbracket^{H,\mu}.
\end{array}
\end{equation*}
\noindent
Given a Heyting algebra $H$, we say that a formula $\phi$ is \emph{valid on $H$} (in symbols $H \vDash \phi$) if for every valuation $\mu$ we have $\sem{\phi}^{H,\mu} = 1$.
Given a class of Heyting algebras $\class$, we say that $\phi$ is \emph{valid on $\class$} (in symbols $\class \vDash \phi$) if $\phi$ is valid on every member of $\class$. We call the set of formulas valid on the class $\class$ the \emph{logic of $\class$} and we write $\Log(\class)$. It is well known that the logic of $\HA$ is  $\ipc$.

We say that a set of formulas $L$ in the signature $\langInt$  is an \emph{intermediate logic} if $\ipc\subseteq L\subseteq \cpc$  and, additionally, $L$ is closed under \emph{modus ponens} and uniform substitution. We shall write $L\vdash \phi$ when $\phi\in L$. A possibly surprising result is that not only $\ipc$, but \emph{every} intermediate logic is sound and complete with respect to a variety of Heyting algebras \cite{Zakharyaschev.1997}. 

This result can be combined with Theorem \ref{theorem:esakiaDuality} to obtain a semantics based on Esakia spaces. Let $\espace{E}$ be an Esakia space and consider a map $\mu: \at \to \clop\upset(\espace{E})$, which we call a \emph{topological valuation}. We can define an interpretation of formulas of $\langInt$ based on \emph{clopen upsets} of $\espace{E}$: 
\begin{equation*}
\begin{array}{r@{\hspace{.1em}}l @{\hspace{1em}}  r@{\hspace{.1em}}l @{\hspace{1em}}  r@{\hspace{.1em}}c@{\hspace{.1em}}l}	
\llbracket p \rrbracket^{\espace{E},\mu} &= \mu(p)
&\llbracket \bot \rrbracket^{\espace{E},\mu} &= \emptyset \\[0.5em]
\llbracket \top \rrbracket^{\espace{E},\mu} &=  \mathfrak{E} 
&\llbracket \phi \land \psi \rrbracket^{\espace{E},\mu} &= \llbracket \phi \rrbracket^{^{\espace{E},\mu}} \cap \llbracket \psi \rrbracket^{\espace{E},\mu}  \\[0.5em]
\llbracket \phi \rightarrow \psi \rrbracket^{\espace{E},\mu} &= \comp{\llbracket \phi \rrbracket^{\espace{E},\mu} \setminus \llbracket \psi \rrbracket^{\espace{E},\mu}}
&\llbracket \phi \lor \psi \rrbracket^{\espace{E},\mu} &= \llbracket\phi \rrbracket^{^{\espace{E},\mu}} \cup  \llbracket\psi \rrbracket^{^{\espace{E},\mu}}.
\end{array}
\end{equation*}

\noindent
Notice that these are exactly the Heyting algebra operations of the dual algebra $H_{\espace{E}}$. We say that a formula $\phi$ is valid on a space $\espace{E}$ (in symbols $\espace{E} \vDash \phi$) if for every valuation $\mu:\at \to \clop\upset(\espace{E})$ we have $\sem{\phi}^{\espace{E},\mu} = \espace{E}$. 	For a class $\mathcal{E}$ of Esakia spaces, we say that $\phi$ is valid on $\mathcal{E}$ (in symbols $\mathcal{E} \vDash \phi$) if $\phi$ is valid on every member of the class. We call the set of formulas valid on the class $\mathcal{E}$ the \emph{logic of $\mathcal{E}$} and we write  $\Log(\mathcal{E})$. We stress that, in the literature on intuitionistic logic, the term \emph{topological semantics} refers to a different semantics from the one presented here, one in which atomic formulas are assigned to \textit{opens} of an \textit{arbitrary} topological space (see \cite{bezhanishvili2019semantic}).

As a consequence of the results for classes of Heyting algebras, we have that the logic of the class $\Esa$ of all Esakia spaces is intuitionistic logic and that every intermediate logic is the logic of some class of Esakia spaces. Firstly, let us recall the correspondence between varieties of Heyting algebras and intermediate logics.	Let $L$ be an intermediate logic and $Var(L)=\{H\in \HA \mid H\vDash L \}$ the corresponding variety.	The algebraic completeness theorem for intermediate logics states that for any intermediate logic $L$, $L\vdash \phi $ if and only if $Var(L)\vDash \phi$.
Conversely, if $\mathcal{V}$ is a variety of Heyting algebras, then the definability theorem of varieties of Heyting algebra tells us that $H\in\vari$ if and only if $H\vDash Log(\mathcal{V})$, where $ Log(\mathcal{V}) = \{ \phi\in\langInt \,|\, \mathcal{V}\vDash\phi  \}$ is the logic of $\mathcal{V}$.

Using the Esakia duality and the semantics presented above, it follows that $H\vDash \phi$ if and only if $\mathfrak{E}_H\vDash \phi$. We can then translate the definability theorem and the algebraic completeness to the setting of Esakia spaces. To this end, we firstly define the concept corresponding to a variety of Heyting algebras: 	we say that a class of Esakia spaces $\mathcal{E}$ is a \textit{variety of Esakia spaces} if $\mathcal{E}$ is closed under p-morphic images, closed upsets and coproducts. These operations correspond through Esakia duality to the operations of subalgebras, homomorphic images and products respectively.\footnote{Finite coproducts of Esakia spaces are simply disjoint unions thereof, while infinite coproducts require additionally that one takes a suitable compactification of infinite disjoint unions, see e.g. \cite[Ex. 5.3.11]{gehrke2023topological}.}

Let $\Lambda(\HA)$ and $\Lambda(\ipc)$ be the complete lattice of varieties of Heyting algebras and the complete lattice of intermediate logics respectively (see \cite[Sec. 7.6]{Zakharyaschev.1997}). 	It is straightforward to show that the arbitrary intersection of varieties of Esakia spaces is again a variety, thus we have that the family of varieties of Esakia spaces $\Lambda(\Esa)$ forms a complete lattice too. In analogy with the algebraic case, we define the two functions $Space: \Lambda(\ipc) \rightarrow  \Lambda(\Esa)$ and $Log: \Lambda(\Esa) \rightarrow  \Lambda(\ipc)$ as follows:
\begin{align*}
Space(L)&= \{\mathfrak{E}\in \Esa \mid  \mathfrak{E}\vDash L  \};\\
Log(\mathfrak{\mathcal{E}})&=\{\phi\in\langInt \mid  \mathcal{E}\vDash \phi  \}.
\end{align*}

\noindent
By looking at these maps in light of the duality between Heyting algebras and Esakia spaces, we readily obtain the following result, which establishes a version of completeness and definability for varieties of Esakia spaces.

\begin{theorem}
	Let $L$ be an intermediate logic,  $\mathcal{E}$ a variety of Esakia Spaces, $\phi$ a formula and  $\mathfrak{E}$ an Esakia Space. Then we have the following:
	\begin{align*}
	\phi \in L &\Longleftrightarrow Space(L)\vDash \phi; \\
	\mathfrak{E}\in \mathcal{E} & \Longleftrightarrow \mathfrak{E}\vDash Log(\mathcal{E}).
	\end{align*}
\end{theorem}

Finally, Esakia duality can be lifted to the level of the lattices of varieties of Heyting algebras and of Esakia spaces. In particular, the maps
\begin{equation*}
\begin{array}{l@{\hspace{3em}}l}
\overline{\PF}: \Lambda(\HA) \to \Lambda(\Esa)
&\overline{\clop\upset}: \Lambda(\Esa) \to \Lambda(\HA) \\
\overline{\PF}(\vari) =  \{ \fE \mid \fE\cong\espace{E}_H \text{ for }  H\in \vari \}
&\overline{\clop\upset}(\mathcal{E}) = \{ H\mid H \cong H_{\espace{E}} \text{ for }  \espace{E} \in \mathcal{E} \}
\end{array}
\end{equation*}

\begin{wrapfigure}{r}{6.3cm}		
	\adjustbox{scale=0.8,}{
		\begin{tikzcd}
			\Lambda(\HA) \arrow[rrddd, "Log"', bend right] \arrow[rrrr, "\overline{\PF}", bend left] &  &                                                                 &  & \Lambda(\Esa) \arrow[llddd, "Log", bend left] \arrow[llll, "\overline{\clop\upset}"] \\
			&  &                                                                 &  &                                                                            \\
			&  &                                                                 &  &                                                                            \\
			&  & 	\Lambda(\ipc)^{op} \arrow[lluuu, "Var"'] \arrow[rruuu, "Space"] &  &                                                                           
	\end{tikzcd}}
\end{wrapfigure}

\noindent are inverse to each other.  The names $\overline{\PF}$ and $\overline{\clop\upset}$ indicate that these maps can be seen as liftings of the maps $\PF$ and $\clop\upset$ respectively to varieties. The lattices $\Lambda(\HA)$ and $\Lambda(\Esa)$ are then isomorphic, whence we also obtain that $\Lambda(\ipc)\cong^{op}\Lambda(\HA)\cong\Lambda(\Esa)$. The relations between the lattices $\Lambda(\HA)$, $\Lambda(\Esa)$  and $\Lambda(\ipc)$ are depicted in the diagram to the right, where arrows indicate lattice isomorphisms.

\subsection{$\dna$-Logics}
\label{subsection:dnaLogicsAndAlgebraicSemantics}

In this paper, we are especially interested in regular Heyting algebras and their connection to Esakia spaces. This class of structures has important connections to a family of (non-standard) logics closely related to intermediate logics, i.e. \emph{$\dna$-logics}, from \textit{double negation on atoms}. These logics were originally introduced in \cite{Miglioli1989-PIESRO} and later studied in \cite{Ciardelli.2009,bezhanishvili_grilletti_quadrellaro_2022}. Notice that, for any formula $\phi$, we write $\phi[\sfrac{\overline{\neg p}}{\overline{p}}]$ for the formula obtained by replacing simultaneously $p$ by $\neg p$ for any atom $p$ occurring in $\phi$.

\begin{definition}
	For every intermediate logic $L$,  its \textit{negative variant} $L^\neg$ is
	$$ L^\neg \;=\; \{\,\phi\in\langInt \,|\, \phi[\sfrac{\overline{\neg p}}{\overline{p}}]\in L  \,\}.$$
	We call the negative variant of some intermediate logic a \emph{$\dna$-logic}.
\end{definition}

\noindent	Every $\dna$-logic contains the formula $\neg\neg p \to p$ for every atomic proposition $p\in\at$---but in general this is not true if we replace $p$  by an arbitrary formula $\phi$.
So we can think of $\dna$-logics as intermediate logics where atoms do not play the role of \emph{arbitrary formulas}, since the principle of \emph{uniform substitution} does not hold, but rather the role of \emph{arbitrary negated formulas}. We notice that $\dna$-logics are an example of \textit{weak logics} in the sense of \cite[Def. 2]{nakov.quadrellaro}, i.e. they are consequence relations closed under permutations of atomic variables. 

Given a $\dna$-logic $\dnal{L}$ there is a standard way to find an intermediate logic $L$ such that $L^{\neg} = \dnal{L}$, as the following lemma shows.	
\begin{lemma}
	Given a $\dna$-logic $\dnal{L}$, define the set
	\begin{equation*}
	S(\dnal{L}) \;:=\; \{\, \phi \,|\, \sigma (\phi) \in \dnal{L} \text{ for every substitution $\sigma$} \,\}.
	\end{equation*}
	Then $S(\dnal{L})$ is an intermediate logic and $(S(\dnal{L}))^{\neg} = \dnal{L}$.
\end{lemma}

\noindent We refer the reader to \cite[Thm. 4.6]{bezhanishvili_grilletti_quadrellaro_2022} for the proof of the previous lemma. 	$S(\dnal{L})$ is usually referred to as the \emph{schematic fragment} of the logic $\dnal{L}$---see for example \cite{Ciardelli.2009}.

There is also another way to characterize $\dna$-logics, that is, through their \emph{algebraic semantics} based on Heyting algebras.  	Let $H$ be a Heyting algebra, then  we call a valuation $\mu: \at \to H$ \emph{negative} if every atom is mapped to a regular element of $H$, or equivalently if $\neg\neg \mu(p) = \mu(p)$ for every $p$. If we restrict the algebraic semantics presented in Subsection \ref{subsection:semanticsIntermediate} to negative valuations we obtain a correct semantics for $\dna$-logics, in the following sense:
If $H$ is a Heyting algebra, the set of formulas $\phi$ such that $\sem{\phi}^{H,\mu} = 1$ for every \emph{negative valuation} $\mu$ is a $\dna$-logic---we call this set the \emph{$\dna$-logic of $H$}. We write $ H\vDash^{\neg} \phi $ if $\sem{\phi}^{H,\mu} = 1$ for every \emph{negative valuation} $\mu$, and we extend this notion to classes of algebras in the usual way.

In \cite{bezhanishvili_grilletti_quadrellaro_2022} this semantics was employed in order to adapt results from the field of intermediate logic to study $\dna$-logics. In particular, we can show that $\dna$-logics form a lattice $\Lambda(\ipc^\neg)$, dual to a particular sublattice of $\Lambda(\HA)$. We write $ K \preceq H$ whenever $K$ is a subalgebra of $H$. 

\begin{definition}[$\dna$-variety]\label{dna-variety}
	A variety of Heyting algebras $\vari$ is called a \emph{$\dna$-variety} if it is additionally closed under the operation:
	\begin{equation*}
	\vari^{\uparrow} \;=\; \{\, H \mid \exists K\in \vari.\; K_{\neg} = H_{\neg} \text{ and } K \preceq H \,\}.
	\end{equation*}
\end{definition}

\noindent We write $\dvari(\cC)$ for the smallest $\dna$-variety of algebras containing $\cC$. We let $\Lambda(\HA^{\uparrow})$ be the sublattice of $\Lambda(\HA)$ comprised of all and only the $\dna$-varieties. We remark here that, since $\dna$-varieties are uniquely determined by their regular elements, there is a one-to-one correspondence between $\dna$-varieties and varieties generated by regular Heyting algebras.	

It can be shown \cite[\S 3.4]{bezhanishvili_grilletti_quadrellaro_2022} that for any $\dna$-logic $\mathtt{L}$ the set  $Var^{\neg}(\mathtt{L}) := \{ H \mid  \forall \phi \in \dnal(L).\; H\vDash^{\neg} \phi \}$ is a $\dna$-variety and that given $\vari$ a $\dna$-variety, the set  $Log^{\neg}(\vari) := \{ \phi \mid  \vari\vDash^{\neg} \phi \}$ is a $\dna$-logic. With these preliminary results in place, we can state the correspondence between $\dna$-logics and $\dna$-varieties, analogous to the one for intermediate logics and varieties. See \cite[Thm 3.35]{bezhanishvili_grilletti_quadrellaro_2022} for the proof of the following theorem. 
\begin{theorem}
	The lattices $\Lambda(\ipc^{\neg})$ and $\Lambda(\HA^{\uparrow})$ are dually isomorphic.
	In particular, the maps $Var^{\neg}$ and $Log^{\neg}$ are inverse to each other.
\end{theorem}

The (propositional) inquisitive logic $\inqB$  \cite{Ciardelli2022-lp, ciardelli2018inquisitive}  is usually introduced, analogously to dependence logic, in terms of team semantics (see \cref{dependence.logic}). However, it can also be viewed as a $\dna$-logic. We start by recalling the definitions of the following intermediate logics:
\begin{equation*}
\begin{array}{l @{} l}
\text{\texttt{KP} }	& = \space  \ipc \space + \space (\neg p \rightarrow q\lor r) \rightarrow (\neg p \rightarrow q) \lor (\neg p \rightarrow r) \\
\text{\texttt{ND} }	& =\space  \ipc \space + \space \{ (\neg p \rightarrow \bigvee_{i\leq k} \neg q_i )\rightarrow \bigvee_{i\leq k}(\neg p \rightarrow \neg q_i) \mid k\geq 2 \}.
\end{array}
\end{equation*}

\noindent Moreover, we define the intermediate logic $\mathtt{ML}$ as the set of all formulas which are valid in posets of the form $(\wp(n){\setminus}\emptyset, \supseteq)$ for $0<n<\omega$, under the usual Kripke semantics. It is a well-known fact that $\mathtt{ND}\subseteq \mathtt{KP}\subseteq \mathtt{ML}$ (see e.g.  \cite{Zakharyaschev.1997}). The following theorem establishes an important connection between these intermediate logics and inquisitive logic.

\begin{theorem}[Ciardelli \cite{Ciardelli.2009}]		
	Inquisitive logic is the negative variant of any intermediate logic $L$ such that $\mathtt{ND}\subseteq L\subseteq \mathtt{ML}$.
\end{theorem}

In the light of the previous theorem, the algebraic approach that we introduced  to study $\dna$-logics can be employed to study inquisitive logic as well, as it was done in \cite{bezhanishvili_grilletti_quadrellaro_2022}. In fact, such algebraic approach  	extends  the original work from \cite{grilletti} on the algebraic semantics of inquisitive logic. It was later shown in \cite{nakov.quadrellaro} that this algebraic semantics for $\inqB$ is unique, making $\inqB$ (as well as every $\dna$-logic) algebraizable in a suitable sense.

\section{The Stone Space of Maximal Elements}
\label{subsection:stoneSpaceMaximalElements}
We start by introducing and recalling some basic properties of regular clopens of Esakia spaces. These properties belong to the folklore, but we shall provide details of the proofs in these sections as it does not seem to us that they are explicitly presented in the past literature. We stress however that \cref{hmap} is already stated in \cite[A.2.1]{9783030120955} and \cite[\S 3]{10.2307/20016257}.

First, notice that given an Esakia space $\espace{E}$ we can consider two topologies on it:	the equipped Esakia topology $\tau_{\espace{E}}$ and the Alexandrov topology $\tau_\leq$ induced by the partial order on $\espace{E}$, i.e., the topology having upsets as open sets. To distinguish the interior and closure operators in the two topologies we use the notations $\Int$, $\Cl$ and $\Int_{\leq}, \Cl_{\leq}$ respectively. As the next definition makes explicit, in the rest of this article when we speak of regular subsets of an Esakia space we always mean \emph{regular sets under the order topology}.	
\begin{definition}
	An upset $U$ of an Esakia space $\espace{E}$ is \textit{regular} if ${\Int_{\leq}(\Cl_{\leq}(U)) = U}$.
\end{definition}	
\noindent	We denote by $\mathcal{UR}(\espace{E})$ the regular upsets of $\espace{E}$,  and we denote by $\mathcal{RCU}(\espace{E})$ the set of upsets of $\espace{E}$ that are (i) regular according to the Alexandrov topology and (ii) clopen according to the equipped Esakia topology.   We start by providing several equivalent characterisations of such subsets. Recall that we let $\overline{U}=((U)^\downarrow)^c$ and that if $U\in\mathcal{CU}(\fE)$ then $\overline{U}$ is its pseudocomplement in the Heyting algebra $\mathcal{CU}(\fE)$.

\begin{proposition}\label{regulars}
	Let $\fE$ be an Esakia space and let $U\in\mathcal{CU}(\fE)$.
	Then the following are equivalent:
	\begin{enumerate}
		\item $  U $ is regular;				
		\item $U=\overline{\overline{U}}$;				
		\item $U^\downarrow\setminus U   \;\subseteq\;   \overline{U}^\downarrow\setminus \overline{U}$.	
	\end{enumerate}
\end{proposition}
\begin{proof}
	Firstly, we notice that just by the definition of closure and interior we immediately obtain the following:
	\begin{align*}
	x\in \Int_{\leq}(\Cl_{\leq}(U))  \Longleftrightarrow x^\uparrow \subseteq U^\downarrow \Longleftrightarrow x\notin ((U^\downarrow)^c)^\downarrow \Longleftrightarrow x\in \overline{\overline{U}},
	\end{align*}
	\noindent showing the equivalence of (i) and (ii). The equivalence between (ii) and (iii) is then proved as follows.	    
	\begin{description}
		\item[$(ii)\Rightarrow (iii)$] \medskip Suppose $U=\overline{\overline{U}}$ and let $x\in U^\downarrow\setminus U$.
		Since $x\notin U = (\overline{U}^{\downarrow})^c$, it follows that $x \in \overline{U}^{\downarrow}$.
		Moreover, since $x\in U^\downarrow$, we have that $x\notin ( U^\downarrow)^c = \overline{U}$.
		That is, $x\in  \overline{U}^\downarrow\setminus \overline{U}$.
		
		\item[$(iii)\Rightarrow(ii)$] Suppose that $U^\downarrow\setminus U \,\subseteq\, \overline{U}^\downarrow\setminus \overline{U}$ , we want to show that $U = (\overline{U}^{\downarrow})^c$. 
		
		($\subseteq$)
		Take $x \in U$ and consider any $y\geq x$, which lies again in $U$ since it is an upset.
		Since $U \cap \overline{U} = \emptyset$ it follows that $y \notin \overline{U}$;
		and since $y$ is an arbitrary element above $x$, it follows that $x\notin \overline{U}^{\downarrow}$, that is, $x \in (\overline{U}^{\downarrow})^c$.
		
		$(\supseteq)$
		Now suppose $x \in (\overline{U}^{\downarrow})^c$, which entails $x\notin \overline{U}^\downarrow \setminus \overline{U} $.
		Thus by assumption we have that $x\notin U^\downarrow \setminus U $.
		Then, either $x\in U$, which proves our claim, or $x\notin U^\downarrow$.
		However the latter gives a contradiction, since $x\in (U^\downarrow)^c = \overline{U}$ contradicts our assumption that $x\in  (\overline{U}^{\downarrow})^c \subseteq \overline{U}^c$.
		Hence we have that $x\in U$, which proves our claim. \qedhere
	\end{description}
\end{proof}		

\noindent	Now, given an upset $Q$ we indicate with $M(Q)$ the set of \emph{maximal} elements of $Q$, that is:
\begin{equation*}
M(Q)  \;:=\;  \{ q\in Q  \mid   \forall q' \in Q.( q' \geq q \implies q' = q )  \}.
\end{equation*}

\noindent We often write simply $M(p)$ in place of $M(p^\uparrow)$. We especially remark that, by compactness, it follows that for every Esakia space $\fE$ and for every element $x\in \fE$ the set $M(x)$ is nonempty -- see e.g. \cite[Thm. 3.2.1]{9783030120955}. An important characterisation of elements in $\mathcal{RCU}(\espace{E})$ is then in terms of the maximal elements of the Esakia space $\espace{E}$, as the following proposition makes precise. 

\begin{proposition}\label{proposition:point condition regular upsets}
	Let $\fE$ be an Esakia space and $U\in\mathcal{CU}(\fE)$.
	Then the following are equivalent:
	\begin{enumerate}
		\item $U$ is regular;
		\item For every $x\in \espace{E}$ we have that $x\in U$ if and only if $ M(x)\subseteq U$.
	\end{enumerate}
\end{proposition}	
\begin{proof}
	Firstly notice that, if $x \in U$ then $M(x) \subseteq U$ since $U$ is an upset.
	So in particular $(ii)$ boils down to the right to left direction. We prove the two implications $(i) \Rightarrow (ii)$ and $(ii) \Rightarrow (i)$ separately.
	\begin{description}
		\item[$(i)\Rightarrow (ii)$]
		Given $x \in \espace{E}$, suppose that $M(x)\subseteq U$;
		we want to show that $x \in U$. Towards a contradiction, assume that $x\notin U$, which together with the previous assumption entails $x \in U^{\downarrow}\setminus U$.
		Since $U$ is regular by assumption, by Proposition \ref{regulars} it follows that $x\in \overline{U}^\downarrow\setminus \overline{U}$.
		Since $\overline{U}$ is an upset itself and, by the remark above, there are maximal elements above every point of an Esakia space, it follows that $M(x) \cap \overline{U} \ne \emptyset$.
		But this is in contradiction with $M(x) \subseteq U$ since $U \cap \overline{U} = \emptyset$. 
		
		\item[$(ii)\Rightarrow (i)$]
		By Proposition \ref{regulars}, it suffices to show that if $x \in U^{\downarrow} \setminus U$ then $x \in \overline{U}^{\downarrow} \setminus \overline{U}$.
		So consider $x \in U^{\downarrow} \setminus U$.
		Since $x\notin U$, by assumption $M(x)\nsubseteq U$.
		By maximality of the elements in $M(x)$, we have that  $M(x)\nsubseteq U^\downarrow$, thus $M(x) \cap \overline{U} \ne \emptyset$.
		This implies that $x \in \overline{U}^{\downarrow}$.	Moreover, since  $x\in U^\downarrow$ we have that $x\notin \overline{U}$, thus concluding that $x \in \overline{U}^\downarrow \setminus \overline{U}$. \qedhere
	\end{description}
\end{proof}

By \cref{regulars} the elements of $\mathcal{RCU}(\mathfrak{E})$ correspond one-to-one to the regular elements of $H_{\espace{E}}$. In the light of this fact, it immediately follows that $ \mathcal{RCU}(\mathfrak{E}) $ is a Boolean algebra, where negation is defined as $\neg U:= \overline{U}$ and disjunction as $U\dot{\lor} V:= \neg (\overline{U}\cap \overline{V})$. \cref{proposition:point condition regular upsets} suggests then a connection between the Boolean algebras of regular elements and the \emph{Stone space of the maximal elements of $\espace{E}$}. Consider the set $M_{\espace{E}}$ of maximal elements of $\espace{E}$. It is well-known \cite[Thm. 3.2.3]{9783030120955} that this set forms a Stone space under the relative topology $\tau_{M_\espace{E}}$ inherited from $\espace{E}$:
\begin{equation*}
U\in \tau_{M_\mathfrak{E}}  \;\Longleftrightarrow\;  \exists V \in \tau_{\mathfrak{E}} \text{ such that } U= V\cap  M_\mathfrak{E}.
\end{equation*}

\noindent The following theorem provides a correspondence between $\mathcal{RCU}(\espace{E})$ and $\cC(M_\espace{E})$. We attribute this result to Esakia, as it is mentioned in \cite[A.2.1]{9783030120955}, but we develop the proof idea from \cite[\S 3]{10.2307/20016257}.

\begin{theorem}[Esakia]\label{hmap}
	Let $\espace{E}$ be an Esakia space and $\mathcal{C}(M_{\espace{E}})$ the clopen sets of the Stone space $M_{\espace{E}}$.
	Then the map:	
	\begin{align*}
	M: & \: \mathcal{RCU}(\mathfrak{E}) \rightarrow	\mathcal{C}(M_\mathfrak{E}) \\
	M: & \: U \mapsto U \cap M_\mathfrak{E}
	\end{align*}
	\noindent is an isomorphism of Boolean algebras.
\end{theorem}
\begin{proof}
	First, we show that $M$ is well-defined:
	let $U\in \mathcal{RCU}(\espace{E})$.
	Since $U$ is a clopen of $\espace{E}$, then $M(U) = U\cap M_\mathfrak{E}$ is a clopen of $M_{\espace{E}}$ by definition of the relative topology.
	
	Secondly, we check that $M$ is a homomorphism.
	The only non-trivial case to check is the condition for negation.
	Let $U \in \mathcal{RCU}(\espace{E})$, then we have:
	\begin{equation*}
	M(\neg U) \;=\;
	M(\, \overline{U} \,) \;=\;
	M_\mathfrak{E} \cap (U^{\downarrow})^c \;=\;
	M_\mathfrak{E}\setminus M(U) \;=\;
	\neg M(U);
	\end{equation*}
	where the latter negation is computed in the Boolean algebra $\mathcal{C}(M_\mathfrak{E}) $.
	
	Thirdly, we show that $M$ is injective.
	Suppose $M(U) = M(V)$ for $U,V \in \mathcal{RCU}(\espace{E})$, then for any $x\in \espace{E}$ we have that $M(x)\subseteq U$ if and only if $M(x)\subseteq V$.
	So by Proposition \ref{proposition:point condition regular upsets} it follows that $x\in U$ if and only if $x\in V$, whence  $U = V$.
	
	Finally, we show that $M$ is surjective.
	Let $U\in \mathcal{C}(M_\espace{E})$, since $U$ is a clopen of $M_{\espace{E}}$ under the relative topology and $M_{\espace{E}}$ is closed in $\fE$, we have by compactness that $U=V\cap M_{\espace{E}}$ for some  $V$ clopen in the Esakia topology of $\fE$. By the display above, we have that for all $W\in \mathcal{CU}(\espace{E})$, $M(\comp{W})=M_\fE\setminus M(W)$, from which it follows that
	\[ M(\comp{\comp{V}}) = M_\fE\setminus(M_\fE\setminus M(V))=M_\fE\cap V =U. \]
	Since  $\comp{\comp{V}}\in \mathcal{RCU}(\espace{E}) $ this shows that $M$ is also surjective.
\end{proof}

\noindent The following corollary follows immediately using Stone duality. Notice that if $B$ is a Boolean algebra we write $\fS_B$ for its dual Stone space.

\begin{corollary}\label{stonemaxdual}
	Let $H$ be a Heyting algebra, then the Stone dual $\fS_{H_\neg}$ of the Boolean algebra $H_\neg $ is isomorphic to the Stone space $M_{\espace{E}_H}$, i.e. $ \fS_{H_\neg}\cong M_{\espace{E}_H}$.
\end{corollary}
\begin{proof}
	By Theorem \ref{hmap} we have that $\mathcal{RCU}(\espace{E}_H)\cong \mathcal{C}(M_{\espace{E}_H})$, and consequently $H_\neg\cong\mathcal{C}(M_{\espace{E}_H})$.
	By Stone duality it follows that $ \fS_{H_\neg}\cong M_{\espace{E}_H} $.
\end{proof}

\section{Regular Esakia Spaces}\label{sec.characterisation}

A main goal of this work is to study Esakia spaces dual to regular Heyting algebras. We start by giving them a name.

\begin{definition}
	An Esakia space $\fE$  is \emph{regular} if $H_\mathfrak{E}= \langle  (H_{\mathfrak{E}})_\neg \rangle $.
\end{definition}

\noindent Given an Esakia space $\fE$, we also write $\mathfrak{E}_r$ for the Esakia space dual to $\langle (H_{\mathfrak{E}})_\neg \rangle$.  The notion of regular Esakia spaces is thus defined in \textit{external} terms, by means of Esakia duality. In this section we consider the problem of providing an \textit{internal} characterisation of  regular Esakia spaces.

We give two partial answers to this question. Firstly, in \cref{morphism.characterisation}, we give a characterisation of regular Esakia spaces in terms of special p-morphisms, and we apply it to the finite case to obtain a more fine-grained description. Secondly, in \cref{quotient.characterisation}, we follow an alternative approach in terms of suitable equivalence relations.  This allows us to obtain a necessary condition for an Esakia space to be regular and also a description of finite regular posets. Finally, in  \cref{results_regular_ipc}, we use duality methods to prove some additional results on varieties generated by regular Heyting algebras.

\subsection{A Characterisation by Regular-Preserving Morphisms} \label{morphism.characterisation}

One way to characterise regular Heyting algebras is to look at homomorphisms fixing their Boolean algebra of regular elements. This motivates the following definition. 

\begin{definition} Let $h:\mathfrak{E}\rightarrow \mathfrak{E}'$ be a p-morphism, then \emph{$h$ preserves regulars} if $ h^{-1}: \mathcal{RCU}(\mathfrak{E}')  \to \mathcal{RCU}(\mathfrak{E})$ is an isomorphism of Boolean algebras.
\end{definition}

\noindent As regular clopen upsets of Esakia spaces correspond to clopens of maximal elements, it follows that regular preserving p-morphisms can be characterised in terms of their action on maximal points. 

\begin{proposition}\label{maximalinjectivity}
	Let $h:\mathfrak{E}\rightarrow \mathfrak{E}'$ be a p-morphism, then $h$ preserves regulars if and only if $h\restriction M_\mathfrak{E}$ is a homeomorphism.
\end{proposition}
\begin{proof} We provide full details for the left to right direction and simply resort to Stone duality for the converse.
	
	$(\Rightarrow)$ Since $h$ is continuous, it suffices to check that it is a bijection.  We first show that $h\restriction M_\mathfrak{E}$ is an injection. Consider two distinct $x,y\in M_{\mathfrak{E}}$, since $M_{\mathfrak{E}}$ is a Stone space there are two disjoint clopen neighbourhoods $U_x, U_y$ of $x$ and $y$ respectively. From Theorem \ref{hmap} it follows that  $M^{-1}(U_x)\cap M^{-1}(U_y) = \emptyset $. Now, since $h$ preserves regulars, it follows that $h^{-1}{\restriction}\mathcal{RCU}(\mathfrak{E}') $ is an isomorphism, whence $M^{-1}(U_x)= h^{-1}(V_x)$ and $M^{-1}(U_y)= h^{-1}(V_y)$ for some $V_x,V_y\in \mathcal{RCU}(\mathfrak{E}')$ such that $V_x\cap V_y=\emptyset$. Then, it follows that $h(x)\in V_x$ and $h(y)\in V_y$, whence $h(x)\neq h(y)$.
	
	Now let $x\in M_{\fE'}$ and consider the family  $\{U_x^i\mid i\in I\}$ of all clopen neighbourhoods of $x$ in $M_{\fE'}$. We notice that, since any two points in $M_{\fE'}$ are separated by a clopen,  $\bigcap_{i\in I} U^i_x=\{x\}$. Since $ h^{-1}: \mathcal{RCU}(\mathfrak{E}')  \to \mathcal{RCU}(\mathfrak{E})$ is an isomorphism of Boolean algebras, it follows by \cref{hmap} that $g:=M\circ h^{-1}\circ M^{-1}$ is an isomorphism between $ \mathcal{C}(M_{\fE'})$ and $ \mathcal{C}(M_{\fE})$. Now, if $\bigcap_{i\in I} g(U^i_x)=\emptyset$, then by compactness there is some finite $I_0\subseteq I$ such that $\bigcap_{i\in I_0} g(U^i_x)=\emptyset$, contradicting $\bigcap_{i\in I_0} U^i_x\neq\emptyset$. Let $y\in \bigcap_{i\in I} g(U^i_x) $, then $y$ is maximal and additionally $h(y)\in M^{-1}(\bigcap_{i\in I} U^i_x)$. Since $h(y)$ must also be maximal and $M(M^{-1}(\bigcap_{i\in I} U^i_x))=\{x\}$, this shows that $h$ is also surjective.
	
	$(\Leftarrow)$ Since $h\restriction M_\mathfrak{E}$ is a bijection, it follows by Stone duality that the map $M\circ h^{-1} \circ M^{-1}$ is an isomorphism between $ \mathcal{C}(M_{\fE'})$ and $ \mathcal{C}(M_{\fE})$. By \cref{hmap} we then have that $ h^{-1}: \mathcal{RCU}(\mathfrak{E}')  \to \mathcal{RCU}(\mathfrak{E})$ is an isomorphism of Boolean algebras, which proves our claim.
\end{proof}

\noindent It is then immediate to conclude that the embedding of a Heyting algebra into one with the same regular elements induces a surjective p-morphism of the dual spaces which is injective on the maximal elements. 

\begin{corollary}\label{corolla}
	Let $A,B \in \HA$, $A\preceq B$ and $A_\neg=B_\neg$, then there is a surjective p-morphism $h:\fE_B\twoheadrightarrow \fE_A$ which is also injective on maximal elements.
\end{corollary}
\begin{proof}
	By Esakia duality, the inclusion $A\preceq B$ induces a p-morphism $h:\fE_B\twoheadrightarrow \fE_A$ defined  by $h: F \mapsto F\cap A$,  where $F\subseteq B$ is any prime filter over $B$. The fact  that $h$ is continuous and surjective already follows from the duality between subalgebras and quotient spaces. By Proposition \ref{maximalinjectivity} above we also have that $h$ is injective on maximal elements.
\end{proof}

\noindent The following theorem provides a characterisation of regular Esakia spaces.

\begin{theorem}\label{theoremregular}
	The following are equivalent, for any Esakia Space $\mathfrak{E}$:
	\begin{enumerate}
		\item $H_{\mathfrak{E}}$ is regular;
		\item For any Heyting algebra $K$, $ K\preceq H_{\mathfrak{E}}$  and $ (H_{\mathfrak{E}})_\neg=K_\neg$ entail $K= H_{\mathfrak{E}}$;
		\item For any Esakia space $\mathfrak{E}'$ and any surjective p-morphism  $f:\mathfrak{E}\twoheadrightarrow \mathfrak{E}'$, if $f{\restriction}  M_\mathfrak{E}$  is a homeomorphism, then $f$ is a homeomorphism.
	\end{enumerate}
\end{theorem}
\begin{proof}
	Claims (i) and (ii) are equivalent by the definition of being regular. We show the equivalence of (ii) and (iii).
	
	$(ii)\Rightarrow (iii)$. Let $ f:\mathfrak{E}\twoheadrightarrow \mathfrak{E}'$ be a surjective p-morphism, then by Esakia duality we have that $f^{-1}[H_{\mathfrak{E}'}]\preceq H_{\mathfrak{E}}$. By Stone duality, if  $f\restriction  M_\mathfrak{E}$ is a homeomorphism, then $M\circ f^{-1} \circ M^{-1}$ is an isomorphism of Boolean algebras and so by  \cref{hmap}  $f^{-1}: \mathcal{RCU}(\mathfrak{E}') \rightarrow \mathcal{RCU}(\mathfrak{E})   $ is an isomorphism. Since  $\mathcal{RCU}(\mathfrak{E})=(H_{\mathfrak{E}})_\neg$ and $\mathcal{RCU}(\mathfrak{E}')=(H_{\mathfrak{E}'})_\neg$, it follows that $(f^{-1}[H_{\mathfrak{E}'}])_\neg=f^{-1}[(H_{\mathfrak{E}'})_\neg]= (H_{\mathfrak{E}})_\neg$, which by (ii) entails  $f^{-1}[H_{\mathfrak{E}'}]=H_{\mathfrak{E}}$. By Esakia duality it follows that $f$ is injective, and so is a homeomorphism.
	
	$(iii)\Rightarrow (ii)$. Let $K\preceq H_{\mathfrak{E}}$ be such that $K_\neg = (H_{\mathfrak{E}})_\neg$. By Corollary \ref{corolla}, there is a surjective p-morphism $h:\fE\twoheadrightarrow \fE_K$ which is also injective on maximal elements. Hence, $h{\restriction}  M_\mathfrak{E}$ is a continuous bijection of Stone spaces and thus a homeomorphism. By (iii) it follows that $h:\mathfrak{E}\twoheadrightarrow \fE_K$ is a homeomorphism of Esakia spaces and, by Esakia duality, we obtain that $K= H_{\mathfrak{E}}$.
\end{proof}

In the finite context the characterisation of the previous theorem can be further strengthened. We recall the following definitions of  $\alpha$-reductions and $\beta$-reductions \cite{de1966connection, bezhanishvili2006lattices}.

\begin{definition}
	Let $\fF$ be a partial order and $x,y\in \fF$ be distinct elements.
	\begin{itemize}
		\item  Suppose $x^\uparrow=y^\uparrow \cup \{x\}$. An \emph{$\alpha$-reduction} is a surjection $h:\fF\to \fF\setminus\{y\}$ such that $h(y)=x$ and $h(z)=z$ whenever $z\neq y$.
		\item  Suppose   $x^\uparrow\setminus \{x\}=y^\uparrow\setminus\{y\}$. A \emph{$\beta$-reduction} is a surjection $h:\fF\to \fF\setminus\{y\}$ such that $h(y)=x$ and $h(z)=z$ whenever $z\neq y$.
	\end{itemize}
\end{definition}	
\noindent Notice that, since $\alpha$-reductions and $\beta$-reductions are not necessarily continuous with respect to the Stone topology of an Esakia space, we have introduced them only with respect to partial orders and not for Esakia spaces. This explains why we will use them to characterize only finite regular Esakia spaces, whose underlying topology is discrete.  In particular, in the finite case we can always look at immediate successors of points of a poset: given any $x\in\fF$, we let $S(x):=\{ y\in \fF \mid  x<y \text{ and } x< z\leq y \Rightarrow z=y  \}$.  We recall that a Heyting algebra is \textit{subdirectly irreducible} if it has a second greatest element.

\begin{theorem}\label{characterisation1}
	$H$ is a finite, regular (subdirectly irreducible) Heyting algebra if and only if $\fE_H$ is a finite, (rooted) poset such that:
	\begin{itemize}
		\item[(i)] For all non-maximal $x\in \fE_H$, $|S(x)|\geq 2$.		
		\item[(ii)] For all non-maximal $x,y\in \fE_H$, if $x\neq y$, then $S(x)\neq S(y).$
	\end{itemize}
\end{theorem}
\begin{proof}
	It suffices to consider conditions (i) and (ii) since it is already well-known that finite (subdirectly irreducible) Heyting algebras correspond to finite (rooted) posets under Esakia duality. 
	
	$(\Rightarrow)$	(i) If this is not the case, then there is a point $x\in \fE_H$ such that $S(x)=\{y\}$. Then,  we can apply the $\alpha$-reduction $h$ such that $h(x)=h(y)$ and $h(z)=z$ for all $z\neq x$. Then $h$ is a p-morphism which is injective on maximal elements, hence, by \cref{theoremregular}, $h$ is an isomorphism, contradicting $x\neq y$, $h(x)=h(y)$. (ii) If this is not the case, then there are two distinct $x,y\in\mathfrak{E}$ such that $S(x)=S(y)$. We then apply the $\beta$-reduction $h$ such that $h(x)=h(y)$ and $h(z)=z$ for all $z\neq y$. But then $h$ is a  p-morphism which is injective on maximal elements hence, by \cref{theoremregular}, $h$ must also be an isomorphism, which gives us a contradiction.		
	
	$(\Leftarrow)$ If $H$ is not regular, then $\langle H_\neg \rangle \preceq H$ and $\langle H_\neg \rangle \neq H$. By  \cref{corolla} there exists a p-morphism $h: \fE_H\twoheadrightarrow \mathfrak{E}_{\langle H_\neg \rangle}$ which is also injective on maximal elements. Moreover, since $\langle H_\neg \rangle \neq H$, we also have that $\fE_H\neq \mathfrak{E}_{\langle H_\neg \rangle}$, meaning that $h$ is not injective. Since  $\fE_H$ is finite, it follows that $h=f_0\circ\dots\circ f_n$, where each $f_i$   is either an $\alpha$- or a $\beta$-reduction -- see  \cite[Prop. 3.1.7]{bezhanishvili2006lattices}. In particular, $f_n$ is an $\alpha$- or a $\beta$-reduction over  $\fE_H$, meaning that either (i) or (ii) holds.
\end{proof}

\subsection{A Characterisation by Equivalence Relations} \label{quotient.characterisation}

Although \cref{theoremregular} and \cref{characterisation1} already give us suitable characterisations of (finite) regular Esakia spaces, we provide an alternative description of them in terms  of suitable  equivalence relations. This will make more explicit how finite regular posets are controlled by their maximal elements and will also allow for a finer analysis of polynomials of regular elements. 

To this end we introduce a way to identify points over an Esakia space. If $X\subseteq \fF $ and $\theta$ is an equivalence relation, we let $X/\theta:= \{[x]_\theta \mid x\in X  \} $. We define the following equivalence relations, which can also be seen as a kind of bounded bisimulations (see in particular  \cite{visser1996uniform}).

\begin{definition}
	Let $\fE$ be an Esakia space, we define:
	\begin{align*}
	x \eqrel_0 y \; &\;\Longleftrightarrow \; M(x)=M(y) \\
	x\eqrel_{n+1} y \; &\;\Longleftrightarrow \; x^\uparrow/\eqrel_n = y^\uparrow/\eqrel_n  \\
	\eqrel_{\infty} \; &=\;		 \bigcap_{n\in\omega} \eqrel_n.
	\end{align*}
\end{definition}

\noindent For any $x\in \fE$ we simply write  $[x]_{n}$ and $[x]_\infty$ for its equivalence class over $\eqrel_n$ and $\eqrel_\infty$ respectively. 

\begin{lemma}\label{lemma_inductive}
	Let $\fE$ be an Esakia space and $x,y\in \fE$, then $x\eqrel_{n}y$ entails $x\eqrel_{l}y$ for all $l\leq n<\omega$.
\end{lemma}
\begin{proof}
	By induction on $n\geq 1$.
	\begin{itemize}
		\item Let $x\eqrel_{1}y$ and suppose $x\neqrel_{0}y$. Then there is, without loss of generality, some $z\in M_\fE$ such that $x\leq z$ but $y\nleq z$. Thus for all $w\geq y$, $M(z)\neq M(w)$ and so $z\neqrel_{0}w$, which contradicts  $x\eqrel_{1}y$. 
		\item Let $x\eqrel_{n+1}y$ and suppose $x\neqrel_{l}y$ for some $l\leq n$. Then there is, without loss of generality, some $z\geq x$ such that for all $w\geq y$, $z\neqrel_{l-1}w$. By induction hypothesis it follows that $z\neqrel_{n}w$, contradicting $x\eqrel_{n+1}y$. \qedhere
	\end{itemize}
\end{proof}

The intuitive idea behind the relation $\eqrel_n$ is that it captures the equivalence of two points up to a certain complexity of polynomials (terms) over regular elements. To make this idea precise we recall the following notion of implication rank. The key idea is that the rank of an element of $\langle \mathcal{RCU}(\fE)\rangle$ should indicate ``how hard'' it is to obtain this elements from regular ones. 

\begin{definition}[Implication rank] $\;$
	
	\begin{enumerate}[label=(\alph*)]
		\item 
		Let $\phi$ be a polynomial, we define its \textit{implication rank} $\rank(\phi)$ recursively as follows:
		\begin{enumerate}[label=(\roman*)]
			\item If $\phi$ is a constant or variable, then   $\rank(\phi)=0$;
			\item $\rank(\psi\land \chi)=\text{max}\{\rank(\psi),\rank(\chi)  \};$
			\item $\rank(\psi\lor \chi)=\text{max}\{\rank(\psi),\rank(\chi)  \};$
			\item $\rank(\psi\to \chi)=\text{max}\{\rank(\psi),\rank(\chi)  \}+1$.
		\end{enumerate}
		\item Let $\fE$ be an Esakia space, then for every $U\in \langle \mathcal{RCU}(\fE)\rangle$ we let 
		\[ \rank(U)= \text{min}\{\rank(\phi) \mid \phi(V_0,\dots,V_n)=U \text{ for } V_0,\dots,V_n\in \mathcal{RCU}(\fE)   \}. \]
	\end{enumerate}
\end{definition}

The following lemma characterises the relation $\eqrel_{n}$ over (finite) Esakia spaces and it relates it to the implication rank of polynomials over regular elements.  This result mirrors Visser's classical result on bounded bisimulation \cite[Thms. 4.7-4.8]{visser1996uniform} and Esakia and Grigolia's characterisation of finitely generated Heyting algebras \cite{esakia1977criterion, bezhanishvili2006lattices}, but with the key difference that we restrict attention to polynomials over (possibly infinitely many) regular elements.

\begin{proposition}\label{lemmaquotient} $\;$
	\begin{itemize}
		\item[(i)] Let $\fE$ be an Esakia space and let $H = H_{\fE}$ be its dual Heyting algebra. For all $x,y \in \fE$ such that $x \eqrel_{n} y$, $ x \in U $ if and only if $ y \in U$, for all $U\in \langle H_\neg \rangle $ with  $\rank(U)\leq n $.
		\item[(ii)] Let $\fF$ be a finite poset and let $H = H_{\fF}$ be its dual Heyting algebra. If for all $U\in \langle H_\neg \rangle $ with $\rank(U)\leq n $ we have that $ x \in U $ if and only if $ y \in U$, then  $x \eqrel_{n} y$ for all $x,y \in \fF$.
	\end{itemize}
\end{proposition}
\begin{proof}
	We prove the claim (i) by induction on $n$. 
	\begin{itemize}
		\item Let $n=0$ and suppose  $x \eqrel_{0} y$. Let $U\in \langle H_\neg \rangle $ be such that $\rank(U)=0$, it follows $U\in H_\neg$. By $x \eqrel_{0} y$ we have $M(x)=M(y)$ and so $M(x)\subseteq U$ if and only if $M(y)\subseteq U$. Since $U\in \mathcal{RCU}(\fE)$, it follows by Proposition \ref{regulars} that $x\in U$ if and only if $y\in U$.
		
		\item  Let $n=m+1$ and suppose $x \eqrel_{m+1} y$. If $\rank(U)=k\leq m$ then the claim follows by the induction hypothesis together with \cref{lemma_inductive}. If $\rank(U)=m+1$ we proceed by induction on the complexity of the polynomial $\psi$ of least implication rank for which $U=\psi(V_0,\dots,V_k)$ where $V_i\in \mathcal{RCU}(\fF)$ for all $i\leq k$.
		\begin{itemize}
			\item If $\psi$ is atomic, $\psi= \alpha\land\beta$ or  $\psi=\alpha\lor \beta$, then the claim follows immediately by the induction hypothesis.
			
			\item If $\psi=\alpha\to \beta$, let $V=\alpha(V_1,\dots,V_k)$,  $W=\beta(V_1,\dots,V_k)$, clearly $\rank(V)\leq m$, $\rank(W)\leq m$ and $U= ((V\setminus W)^{\downarrow})^c$. 
			
			We show only one direction as the converse is analogous. Suppose $y\notin ((V\setminus W)^{\downarrow})^c$, then there is some $z\geq y$ such that $z\in V\setminus W$. Since $x \eqrel_{m+1} y$, we have $x^\uparrow/\eqrel_{m} = y^\uparrow/\eqrel_{m}$ and thus there is some $k\geq x $ such that $k \eqrel_{m} z$. By induction hypothesis  $k\in V\setminus W$, showing  $x\notin ((V\setminus W)^{\downarrow})^c$.
		\end{itemize} 
	\end{itemize} 
	
	\noindent We next prove item (ii). We let $\fF$ be a finite poset and we reason by induction.
	\begin{itemize}
		\item Let $n=0$ and suppose $x \neqrel_{0} y$. Without loss of generality we have that $M(y)\nsubseteq M(x)$, hence by \cref{hmap} there is a regular upset $U$ such that $M(x)\subseteq U$ and $M(y)\nsubseteq U$, whence by Proposition \ref{regulars} we have  $x\in U$ and $y\notin U$.
		
		\item Let $n=m+1$. If $x\neqrel_{m+1} y $  then $  x^\uparrow/\eqrel_{m} \neq y^\uparrow/\eqrel_{m}$,  hence (without loss of generality) there is $z\geq x$ such that for all $k\geq y$ we have $z\neqrel_{m} k$. By induction hypothesis, for every $k$, there is either an upset $V_k\in \langle H_\neg \rangle$ such that $z\in V_k$ and $k\notin V_k$, or an upset $U_k\in \langle H_\neg \rangle$ such that $z\notin U_k$ and $k\in U_k$, with $\rank(V_k),\rank(U_k)\leq m$ for all $k\geq y$. We let
		\begin{align*}
		I_0&=\{k\geq y \mid z\in V_k, k\notin V_k    \}\\
		I_1&=\{k\geq y \mid z\notin U_k, k\in U_k    \}.
		\end{align*}
		And we define $Z:=(( \bigcap_{k\in I_0} V_k \setminus  \bigcup_{k\in I_1} U_k )^\downarrow)^c $, i.e. $Z:= \bigcap_{k\in I_0} V_k \to \bigcup_{k\in I_1} U_k$ and thus $\rank(Z)\leq m+1$. Clearly $Z$ is an upset and by definition $Z\in \langle \mathcal{UR}(\fF)\rangle$.
		
		Now, since for every $k\in I_0$ we have $z\in V_k$ and for every $k\in I_1$ we have  $z\notin U_k$, it follows that $z\in \bigcap_{k\in I_0} V_k \setminus  \bigcup_{k\in I_1} U_k$, whence $x\notin Z $. 
		
		Then, if $k\in I_0$ then $k\notin V_k$, whence  $ k\notin \bigcap_{k\in I_0} V_k \setminus \bigcup_{k\in I_1} U_k$. Otherwise, if $k\in I_1$ then $k\in U_k$, whence  $ k\notin \bigcap_{k\in I_0} V_k \setminus \bigcup_{k\in I_1} U_k$. Since $I_0\cup I_1=y^\uparrow$, this shows that  $k\notin \bigcap_{k\in I_0} V_k \setminus \bigcup_{k\in I_1} U_k$ for all $k\geq y$ and so $y\in (( \bigcap_{k\in I_0} V_k \setminus  \bigcup_{k\in I_1} U_k )^\downarrow)^c =Z$, finishing our proof. \qedhere
	\end{itemize}		  
\end{proof}

\noindent We remark in passing that the second claim of the previous proposition can be extended to some specific classes of infinite Esakia spaces. For example, it can be generalized to Esakia spaces  dual to finitely generated Heyting algebras, or to Esakia spaces whose order reducts are \textit{image finite posets}, i.e. posets $\fE$ where $|x^\uparrow|<\omega$ for every $x\in \fE$.

The next proposition follows immediately. Notice that by Esakia duality we treat elements of a finite poset as filters over the dual Heyting algebra and  elements of a finite Heyting algebra as upsets of the dual poset.

\begin{proposition}\label{lemmaquotient2}
	Let $\fF$ be a finite poset and let $H = H_{\fF}$ be its dual Heyting algebra, the following are equivalent for any $x,y \in \fF$:
	\begin{enumerate}
		\item $x \eqrel_{\infty} y$;
		\item $x \cap \langle H_{\neg} \rangle  =  y \cap \langle H_{\neg} \rangle$;
		\item $\forall U \in \langle H_\neg \rangle. [ x \in U \iff y \in U ]$.
	\end{enumerate}
\end{proposition}
\begin{proof} Claims (ii) and (iii) are rephrasing of the same condition under Esakia duality. The equivalence of (i) and (iii) follows  from \cref{lemmaquotient}.
\end{proof}

We can use the relation $\eqrel_\infty$ to supplement \cref{maximalinjectivity} and characterise p-morphisms preserving polynomials of regulars between finite posets. 
\begin{definition}
	Let $h:\mathfrak{E}\rightarrow \mathfrak{E}'$ be a p-morphism,  then we say that \emph{$h$ preserves polynomials of regulars}  if  $ h^{-1}: \langle \mathcal{RCU}(\mathfrak{E}') \rangle \to \langle \mathcal{RCU}(\mathfrak{E})\rangle$ is an isomorphism of Heyting algebras.
\end{definition}

\begin{proposition}\label{preservation-polynomials}
	Let $h:\fF\to\fF'$ be a p-morphism between finite posets, then $h$ preserves polynomials of regulars if and only if  $x\neqrel_{\infty} y$ entails $h(x)\neqrel_{\infty}h(y)$.
\end{proposition}
\begin{proof} 
	$(\Rightarrow)$ Suppose $h:\fF\to\fF'$ is a p-morphism preserving polynomials of regular elements, and let $x,y\in \fF$ be such that $x\neqrel_{\infty} y$. Since $h$ preserves polynomials of regulars, the induced map $ h^{-1}: \langle \mathcal{UR}(\fF) \rangle \to \langle \mathcal{UR}(\fF')\rangle$ is an isomorphism of Heyting algebras. It follows that $x\in h^{-1}(U)$ and $y\notin h^{-1}(U)$ for some $h^{-1}(U)\in  \mathcal{UR}(\fF)$.  Hence, we obtain that $h(x)\in U$ and $h(y)\notin U$ for some $U\in  \mathcal{UR}(\fF)$, which by \cref{lemmaquotient2} proves our claim. $(\Leftarrow)$ Analogous to the previous direction.
\end{proof}

Finally, the next theorem  shows that regular Esakia spaces are stable under $\eqrel_\infty$. Moreover, in the finite setting, it provides us with a second characterisation of finite regular posets. 

\begin{theorem}\label{characterisation.regular} $\;$
	\begin{itemize}
		\item[(i)] Let $\fE$ be a regular Esakia space and $x,y\in \fE$; if  $x\eqrel_\infty y$ then $x=y$. 
		\item[(ii)] Let $\fF$ be a finite poset such that $x\eqrel_\infty y$ entails $x=y$, then $\fF$ is regular.
	\end{itemize}
	In particular, a finite Heyting algebra is regular if and only if its dual poset is stable under $\eqrel_{\infty}$.
\end{theorem}
\begin{proof}
	For claim (i) consider distinct $x,y\in \fE$, then there is a  clopen upset $U\in \mathcal{CU}(\fE)$ such that (without loss of generality) $x\in U$ and $y\notin U$. By regularity, it follows that $U=\psi(V_0,\dots,V_k)$ for some regular clopen upsets $V_i$, $i\leq k$. By \cref{lemmaquotient} it follows immediately that $x\neqrel_\infty y$.
	
	For claim (ii), consider an arbitrary surjective p-morphism  $p:\fF \twoheadrightarrow \fF'$ such that $p\restriction M_\fF$ is a bijection. We first prove by induction on $n<\omega$ that $x\neqrel_{n} y$ entails $p(x)\neqrel_{n}p(y)$. 
	\begin{itemize}
		\item If $x\neqrel_{0} y$, then $M(x)\neq M(y)$. Since $p\restriction M_\fF$ is a bijection, we have by the definition of p-morphism that $M(p(x))\neq M(p(y))$ and so $x\neq y$.
		\item If $x\neqrel_{n+1} y$, there is without loss of generality some $z\geq x$ such that $z\neqrel_{n} w$ for all $w\geq y$. By induction hypothesis we obtain that $p(z)\neq p(w)$ for all  $w\geq y$, and by the definition of p-morphism it follows $p(x)\neq p(y)$.
	\end{itemize}
	Thus we obtain that $p$ preserves polynomials of regulars. Finally, since  for $x,y\in \fF$ we have that $x\neq y$ entails $x\neqrel_\infty y$, this shows that $p$ is itself an injection, and thus by  \cref{theoremregular} we have that $\fF$ is regular.	
\end{proof}	

\noindent We conclude by noticing that, by our previous observations, an equivalent sufficient and necessary characterisation of regular Esakia spaces works in the restricted cases of  image-finite Esakia spaces or of Esakia spaces dual to finitely generated Heyting algebras.

\subsection{Varieties of (Strongly) Regular Heyting algebras }\label{results_regular_ipc}

We conclude this section by providing some side results on varieties generated by regular Heyting algebras. It is in fact natural to consider what is the intermediate logic of all regular Heyting algebras. As a matter of fact, it was proven already in \cite[Cor. 5.2.3]{Ciardelli.2009} that  $S(\ipc^\neg)=\ipc$, which by \cite[Prop. 4.17]{bezhanishvili_grilletti_quadrellaro_2022} means that the variety generated by regular Heyting algebras is the whole variety of Heyting algebras. 

One could then wonder if, by looking at suitable subclasses of regular Heyting algebras, one can obtain proper subvarieties of Heyting algebras. For example one could  consider, for all $n<\omega$, the varieties generated by those Heyting algebras which are stable  under $\eqrel_m$ for $m\geq n$ but not under $\eqrel_m$ for $m<n$. In fact, by the next proposition, we have that the sequence of equivalence relations $\eqrel_n$ does not in general converge to any finite value.

\begin{proposition}\label{dna-rieger}
	There is an Esakia space $\fE$ such that each $\fE/\eqrel_{n}$ is distinct, for each $n<\omega$.
\end{proposition}
\begin{proof}
	Consider the poset $\fR_1$ in \cref{rieger-nishimura} (which adapts \cite[Fig. 4.1]{Quadrellaro.2019}) and provide it with the topology induced by the following subbasis
	\[ \{ x^\uparrow \mid  x\in \fR_1 \} \cup \{(x^\uparrow)^c \mid x\in \fR_1    \}.  \]
	It can be  verified that the resulting space is an Esakia space and, moreover, that  for every $n<\omega$, $a_n\neqrel_{n+1}a_{n+1}$. Therefore, for every $n<m<\omega$ we have that $\fR_1/\eqrel_{n}$ and $\fR_1/\eqrel_{m}$ are distinct. \qedhere
\end{proof}

\begin{figure}
	\begin{tikzpicture}[scale=0.25]
	\pgfmathsetmacro{\NODESIZE}{1.5pt}
	
	-- draw left column of nodes
	\foreach \y in {0,...,6}{
		\node[draw,circle, inner sep=\NODESIZE,fill] (leftnode \y) at (0,2*\y) {};
	}
	-- add two invisible nodes for the descending line at infinity
	\node[draw=none] (leftend1) at  (0,-1) {};
	\node[draw=none] (leftend2) at  (0,-3) {};
	
	-- draw right column of nodes
	\foreach \y in {0,...,6}{
		\node[draw,circle, inner sep=\NODESIZE,fill] (rightnode \y) at (5,2*\y) {};
	}
	\node[draw=none] (rightend1) at (5,-1) {};
	\node[draw=none] (rightend2) at (5,-3) {};
	
	-- draw the ladder itself
	\foreach \y in {0,...,4}{
		\pgfmathparse{\y+2}
		\draw [-] (leftnode \y) to (rightnode \pgfmathresult);
	}
	\foreach \y in {0,...,5}{
		\pgfmathparse{\y+1}
		\draw [-] (rightnode \y) to (leftnode \pgfmathresult);
	}
	
	\node[draw,circle, inner sep=\NODESIZE,fill,yshift=0em] (end) at  (2.5,-4) {};
	
	--construct the vertical bars and descend into infinity
	\draw [-] (leftnode 6) to (leftend1);
	\draw [loosely dashed] (leftend1) to (leftend2);
	\draw [-] (rightnode 6) to (rightend1);
	\draw [loosely dashed] (rightend1) to (rightend2);

	-- label the fuckers    
	\node [left] at (leftnode 6.west)  {$a_0$};
	\node [left] at (leftnode 5.west)  {$a_1$ };
	\node [left] at (leftnode 4.west)  {$a_2 $};
	\node [left] at (leftnode 3.west)  {$a_3$};
	\node [right] at (rightnode 6.east) {$b_0$};
	\node [right] at (rightnode 5.east)  {$b_1$};
	\node [right] at (rightnode 4.east)  {$b_2$};
	\node [right] at (rightnode 3.east)  {$b_3$};

	\node [right] at (end.east)  {$r$};
	-- caption trick
	\node[draw=none] (label) at  (2,-7) {};
	\node [] at (label.south)  {$\fR_0$};
	\end{tikzpicture} 
	\hspace{0.5cm}
	\begin{tikzpicture}[scale=0.25]
	\pgfmathsetmacro{\NODESIZE}{1.5pt}
	
	-- draw left column of nodes
	\foreach \y in {0,...,6}{
		\node[draw,circle, inner sep=\NODESIZE,fill] (leftnode \y) at (0,2*\y) {};
	}
	-- add two invisible nodes for the descending line at infinity
	\node[draw=none] (leftend1) at  (0,-1) {};
	\node[draw=none] (leftend2) at  (0,-3) {};
	
	-- draw right column of nodes
	\foreach \y in {0,...,6}{
		\node[draw,circle, inner sep=\NODESIZE,fill] (rightnode \y) at (5,2*\y) {};
	}
	\node[draw=none] (rightend1) at (5,-1) {};
	\node[draw=none] (rightend2) at (5,-3) {};
	
	-- draw the ladder itself
	\foreach \y in {0,...,4}{
		\pgfmathparse{\y+2}
		\draw [-] (leftnode \y) to (rightnode \pgfmathresult);
	}
	\foreach \y in {0,...,5}{
		\pgfmathparse{\y+1}
		\draw [-] (rightnode \y) to (leftnode \pgfmathresult);
	}
	
	--construct the vertical bars and descend into infinity
	\draw [-] (leftnode 6) to (leftend1);
	\draw [loosely dashed] (leftend1) to (leftend2);
	\draw [-] (rightnode 6) to (rightend1);
	\draw [loosely dashed] (rightend1) to (rightend2);
	
	--add the branching to the left node
	\node[draw,circle, inner sep=\NODESIZE,fill] (leftmost) at (-4,12) {};
	\draw [] (leftnode 5) to (leftmost);
	
	\node[draw,circle, inner sep=\NODESIZE,fill,yshift=0em] (end) at  (2.5,-4) {};
	
	-- label the fuckers    
	\node [left] at (leftnode 6.west)  {$a_0$};
	\node [left] at (leftnode 5.west)  {$a_1$ };
	\node [left] at (leftnode 4.west)  {$a_2 $};
	\node [left] at (leftnode 3.west)  {$a_3$};
	\node [left] at (leftmost.west)  {$c_0$};
	\node [right] at (rightnode 6.east) {$b_0$};
	\node [right] at (rightnode 5.east)  {$b_1$};
	\node [right] at (rightnode 4.east)  {$b_2$};
	\node [right] at (rightnode 3.east)  {$b_3$};
	
	\node [right] at (end.east)  {$r$};
	
	-- caption trick
	\node[draw=none] (label) at  (2,-7) {};
	\node [] at (label.south)  {$\fR_1$};
	\end{tikzpicture} 
	\hspace{0.5cm}
	\begin{tikzpicture}[scale=0.25]
	\pgfmathsetmacro{\NODESIZE}{1.5pt}
	
	-- draw left column of nodes
	\foreach \y in {0,...,6}{
		\node[draw,circle, inner sep=\NODESIZE,fill,yshift=0em] (leftnode \y) at (0,2*\y) {};
	}
	-- add two invisible nodes for the descending line at infinity
	\node[draw=none] (leftend1) at  (0,-1) {};
	\node[draw=none] (leftend2) at  (0,-3) {};
	
	-- draw right column of nodes
	\foreach \y in {0,...,6}{
		\node[draw,circle, inner sep=\NODESIZE,fill] (rightnode \y) at (5,2*\y) {};
	}
	\node[draw=none] (rightend1) at (5,-1) {};
	\node[draw=none] (rightend2) at (5,-3) {};
	
	-- draw 2right column of nodes
	\foreach \y in {1,...,6}{
		\node[draw,circle, inner sep=\NODESIZE,fill] (2rightnode \y) at (8,2*\y) {};
	}
	\node[draw=none] (2rightend1) at (8,-1) {};
	\node[draw=none] (2rightend2) at (8,-3) {};

	-- draw center column of nodes
	\foreach \y in {1,...,6}{
		\node[draw,circle, inner sep=\NODESIZE,fill] (centernode \y) at (-4,2*\y) {};
	}
	\node[draw=none] (centerend1) at (-4,-1) {};
	\node[draw=none] (centerend2) at (-4,-3) {};

	-- draw the ladder itself
	\foreach \y in {0,...,4}{
		\pgfmathparse{\y+2}
		\draw [-] (leftnode \y) to (rightnode \pgfmathresult);
	}
	\foreach \y in {0,...,5}{
		\pgfmathparse{\y+1}
		\draw [-] (rightnode \y) to (leftnode \pgfmathresult);
	}
	\foreach \y in {0,...,5}{
		\pgfmathparse{\y+1}
		\draw [-] (leftnode \y) to (centernode \pgfmathresult);
	}
	
	\foreach \y in {0,...,5}{
		\pgfmathparse{\y+1}
		\draw [-] (rightnode \y) to (2rightnode \pgfmathresult);
	}
	
	--construct the vertical bars and descend into infinity
	\draw [-] (leftnode 6) to (leftend1);
	\draw [loosely dashed] (leftend1) to (leftend2);
	\draw [-] (rightnode 6) to (rightend1);
	\draw [loosely dashed] (rightend1) to (rightend2);
	\draw [loosely dashed] (centerend1) to (centerend2);
	\draw [loosely dashed] (2rightend1) to (2rightend2);
	
	-- label the fuckers    
	\node [left] at (leftnode 6.west)  {$a_0$};
	\node [left] at (leftnode 5.west)  {$a_1$ };
	\node [left] at (leftnode 4.west)  {$a_2 $};
	\node [left] at (leftnode 3.west)  {$a_3$};
	\node [right] at (rightnode 6.east) {$b_0$};
	\node [right] at (rightnode 5.east)  {$b_1$};
	\node [right] at (rightnode 4.east)  {$b_2$};
	\node [right] at (rightnode 3.east)  {$b_3$};
	\node [left] at (centernode 6.west) {$c_0$};
	\node [left] at (centernode 5.west)  {$c_1$};
	\node [left] at (centernode 4.west)  {$c_2$};
	\node [left] at (centernode 3.west)  {$c_3$};
	\node [right] at (2rightnode 6.east) {$d_0$};
	\node [right] at (2rightnode 5.east)  {$d_1$};
	\node [right] at (2rightnode 4.east)  {$d_2$};
	\node [right] at (2rightnode 3.east)  {$d_3$};
	-- caption trick
	\node[draw=none] (label) at  (2,-7) {};
	\node [] at (label.south)  {$\fR_2$};
	
	\node[draw,circle, inner sep=\NODESIZE,fill,yshift=0em] (end) at  (2.5,-4) {};

	\node [right] at (end.east)  {$r$};
	\end{tikzpicture}\caption{}	
	\label{rieger-nishimura}
\end{figure}

We then introduce the following definition.

\begin{definition} $\;$
	\begin{enumerate}
		\item Let $\fE$ be an Esakia space, we say that it is \emph{strongly regular} if for all $x,y\in \fE$ we have that $x\neqrel_{0}y$, i.e. $M(x)\neq M(y)$. 
		\item We say that a Heyting algebra $H$ is \emph{strongly regular} if its dual Esakia space $\fE_H$ is strongly regular.
	\end{enumerate}
\end{definition}

\noindent 	For example, if we provide the poset $\fR_2$ with a topology analogous to the one we assigned to $\fR_1$ in \cref{dna-rieger}, we see that $\fR_2$ makes for a strongly regular Esakia space. Moreover, the map $p:\fR_2\to \fR_0$ defined by letting, for each $i<\omega$, $p(a_i)=a_i$, $p(b_i)=b_i$, $p(c_i)=a_0$, $p(d_i)=b_0$ and $p(r)=r$ is a p-morphism from a strongly regular Esakia space onto the dual of the Rieger-Nishimura lattice. Working on this idea we can strengthen the aforementioned result and establish that the variety generated by   strongly regular Heyting algebras  is the variety of all Heyting algebras. We start by defining the strong regularisation of a finite poset.

\begin{definition}
	Let $\fF$ be a finite poset, the \textit{strong regularisation} of $\fF$ is the poset $\fF^*$ obtained by adding, for each element $x\in \fF$, a new maximal element $x^*$ such that $(x^*)^\downarrow=x^\downarrow\cup \{x^*\}$.
\end{definition}

\noindent It is then possible to prove that every finite poset is a p-morphic image of its strong regularisation, hence showing that strongly regular Heyting algebras generate the whole variety of Heyting algebras.

\begin{theorem}
	The variety generated by strongly regular Heyting algebras is $\HA$.
\end{theorem}
\begin{proof}
	Since $\HA$ has the finite model property, it follows that if $\HA\nvDash \phi$ there is a finite Heyting algebra $H$ such that $H\nvDash\phi$. Now, if $\fF\twoheadrightarrow \fE_H $, it follows by duality that $H\preceq H_{\fF}$. So, since  the validity of formulas is preserved by the variety operations, we obtain that $H_{\fF}\nvDash\phi$. It is thus sufficient to show that every  finite poset is a p-morphic image of a finite strongly regular poset. 
	
	To this end, let $\fF$ be an arbitrary finite poset and $\fF^*$ be its strong regularisation. Clearly $\fF^*$ is also finite. Let $p:\fF^*\to\fF$ be defined by letting $p(x)=x$ for all $x\in \fF$, and $p(x^*)\in M(x)$ for all $x^*\in \fF^*\setminus\fF$, i.e. $p$ assigns each $x^*$ to some maximal elements that it “chooses” from $M(x)$. We check that $p$ is a p-morphism.
	\begin{itemize}
		\item[(i)] \textit{Forth Condition:} If $x\leq y$ for $x,y\in \fF$ then obviously $p(x)\leq p(y)$. If $x\leq y^*$ then $x\leq y$ and thus	$p(x)\leq p(y)\leq p(y^*)$.
		
		\item[(ii)] \textit{Back Condition:} If $p(x^*)\leq y$ then since $p(x^*)$ is maximal we immediately have $y=p(x^*)$, satisfying the condition. Otherwise, if $p(x)\leq y$  and $x\in \fF$, then we have that $p(x)\leq y=p(y)$ and by definition of $p$ also that $x\leq y$.
	\end{itemize}
	\noindent This shows that $p:\fF^*\to \fF$ is a p-morphism, which completes our proof.		
\end{proof}

\section{Cardinality of $\Lambda(\HA^{\uparrow})$} \label{cardinality.sublattice}

In this section we apply the characterisation of regular posets of \cref{sec.characterisation} to show that the lattice of $\dna$-varieties $\Lambda(\HA^{\uparrow})$  and the lattice of $\dna$-logics $\Lambda(\ipc^\neg)$  have power continuum. This solves a question raised in \cite{Quadrellaro.2019,bezhanishvili_grilletti_quadrellaro_2022} and complements the previous result that the sublattice of $\dna$-logics extending $\inqB$ is dually isomorphic to $\omega+1$. Inquisitive logic has thus a special location in the lattice of negative variants, having only countably many extensions. 


\subsection{Jankov's Formulae}\label{jankov_formulas} To prove the uncountability of  $\Lambda(\HA^{\uparrow})$ we adapt to our setting the notion and the method of Jankov's formulae. 	Jankov's formulae were introduced in \cite{Jankov.1963, Jankov.1968} in order to show that the lattice of intermediate logic has the cardinality of the continuum. We recall how to adapt Jankov's formulae to the setting of $\dna$-logics and $\dna$-varieties  \cite{bezhanishvili_grilletti_quadrellaro_2022}. We write $\HArfsi$ for the class of all regular, finite, subdirectly irreducible Heyting algebras.

\begin{definition}
	Let $H\in \HArfsi $, let $0$ be the least element of $H$ and $s$ its second greatest element.
	\begin{itemize}
		\item The \textit{Jankov representative} of $x\in H$ is a formula $\psi_x$ defined as follows:
		\begin{itemize}
			\item[(i)] If $x\in H_\neg$, then $\psi_x=p_x$, where $p_x\in\at$;
			
			\item[(ii)] If $x=\delta_H(a_0,...,a_n)$ with $a_0,...,a_n\in H_\neg$, then $\psi_x=\delta(p_{a_0},...,p_{a_n})$.
		\end{itemize}
		
		\item  The \textit{Jankov $\dna$-formula} $\chi^{\dna}(H)$ is defined as follows:	
		$$\chi^{\dna}(H):= \alpha \rightarrow \psi_s,$$
		\noindent where $\alpha$ is the following formula:
		\begin{align*}
		\alpha= (\psi_0\leftrightarrow \bot) \; \land \; &\bigwedge \{(\psi_a\land \psi_b) \leftrightarrow \psi_{a\land b}\mid a,b\in H    \} \; \land \\
		&\bigwedge \{(\psi_a\lor \psi_b) \leftrightarrow \psi_{a\lor b}\mid a,b\in H  \} \; \land \\
		&\bigwedge \{(\psi_a\rightarrow \psi_b) \leftrightarrow \psi_{a\rightarrow b}\mid a,b\in H  \}.
		\end{align*}
	\end{itemize}
\end{definition}

\noindent As it is generally clear from the context that we are dealing with the $\dna$-version of Jankov's formulae, we write just $\chi(H)$ for the Jankov $\dna$-formula of $H$.   We recall the following result from \cite[Thm. 4.31]{bezhanishvili_grilletti_quadrellaro_2022}. For any $A,B\in \HA$, we write $ A\leq B $ if $A\in \mathbb{HS}(B) $.

\begin{theorem}\label{JT}
	Let $A\in \HArfsi$ and $B\in \HA$ then $B\nvDash^\neg \chi(A) \text{ iff } A\leq B.$
\end{theorem}

\noindent The next proposition adapts to the context of $\dna$-logics Jankov's classical result on intermediate logics and Heyting algebras.

\begin{proposition}\label{antichain.jankov}
	Let $\cC$ be an $\leq$-antichain of finite, regular, subdirectly irreducible Heyting algebras, then for all $\mathcal{I},\mathcal{J}\subseteq \cC$ such that $\mathcal{I}\neq \mathcal{J}$ we have that $Log^\neg(\mathcal{I})\neq Log^\neg(\mathcal{J})$.
\end{proposition}
\begin{proof}
	Since $\mathcal{I}\neq \mathcal{J}$ there is without loss of generality some $H\in\mathcal{I}\setminus \mathcal{J}$. By Theorem $\ref{JT}$ it follows that $H\nvDash^\neg \chi(H)$, thus $\chi(H)\notin Log^\neg(\mathcal{I})$. Since $\cC$ is an antichain, $H\nleq K$ for all $K\in \mathcal{J}$, which by Theorem \ref{JT} gives $K\vDash^\neg \chi(H)$, whence $\chi(H)\in Log^\neg(\mathcal{J})$.
\end{proof}

To prove that the lattice of $\dna$-logics has power continuum it is thus sufficient  to exhibit an infinite antichain of finite, regular, subdirectly irreducible Heyting algebras. Perhaps surprisingly, we can use examples of antichains which are standard in the literature, as they turn out to consist of regularly generated algebras. 

\subsection{Antichain $\Delta_0$}\label{antichain0} We start by introducing the antichain $\Delta_0$---for this example see e.g. \cite[p. 71]{bezhanishvili2006lattices}. This is an  antichain of posets which are all regular, but which, as we shall see, contains for all $n<\omega$ infinitely many elements which are not stable under $\eqrel_n$. For every $n<\omega$ we define the poset $\fF_n$, with domain	
\[  \dom(\fF_n)=\{ r\} \cup \{a_m \mid   m\leq n \} \cup \{b_m \mid  m\leq n \} \cup \{c_m \mid  m\leq n \};  \]
\noindent and such that
\begin{itemize}
	\item $r\leq a_i, b_i, c_i \text{ for all } i\leq n;$
	\item $a_i\leq a_j, a_i\leq b_j \text{ and } c_i\leq c_j, c_i\leq b_j \text{ whenever } j\leq i;$
	\item $b_i\leq a_j \text{ and } b_i\leq c_j \text{ whenever } j\leq i.$
\end{itemize}
\begin{figure}
	\begin{center}
		\begin{tikzpicture}[scale=0.8]
		\node[inner sep=0pt] (0)    at ( -8,0) {$\bullet$};
		\node[inner sep=0pt] (00)  at (-9,1) {$\bullet$};
		\node[inner sep=0pt] (01)  at (-8,1) {$\bullet$};
		\node[inner sep=0pt] (10)  at (-7,1) {$\bullet$};
		
		\draw (0)  -- (00);
		\draw (0)  -- (10);
		\draw (0)  -- (01);
		
		\node[inner sep=0pt] (u)    at ( -4,0) {$\bullet$};
		
		\node[inner sep=0pt] (x0)  at (-5,1) {$\bullet$};
		\node[inner sep=0pt](x1) at ( -5,2) {$\bullet$};
		
		\node[inner sep=0pt](y0)  at (-4,1) {$\bullet$};
		\node[inner sep=0pt](y1) at ( -4,2) {$\bullet$};
		
		\node[inner sep=0pt](z0)  at (-3,1) {$\bullet$};
		\node[inner sep=0pt](z1) at ( -3,2) {$\bullet$};
		
		\draw (u)    -- (x0);
		\draw (u)    -- (y0);
		\draw (u)    -- (z0);
		
		\draw (x0)    -- (x1);
		\draw (z0)    -- (z1);
		\draw (x0)    -- (y1);
		\draw (y0)    -- (z1);
		\draw (y0)    -- (x1);
		\draw (z0)    -- (y1);

		
		\node[inner sep=0pt](p)    at ( 0,0) {$\bullet$};
		
		\node[inner sep=0pt](d0)  at (-1,1) {$\bullet$};
		\node[inner sep=0pt](d1) at ( -1,2) {$\bullet$};
		\node[inner sep=0pt](d2)  at (-1,3) {$\bullet$};
		
		\node[inner sep=0pt](e0)  at (0,1) {$\bullet$};
		\node[inner sep=0pt](e1) at ( 0,2) {$\bullet$};
		\node[inner sep=0pt](e2)  at (0,3) {$\bullet$};
		
		\node[inner sep=0pt](f0)  at (1,1) {$\bullet$};
		\node[inner sep=0pt](f1) at ( 1,2) {$\bullet$};
		\node[inner sep=0pt](f2)  at (1,3) {$\bullet$};
		
		\draw (p)    -- (d0);
		\draw (p)    -- (e0);
		\draw (p)    -- (f0);
		
		\draw (d0)    -- (d1);
		\draw (d1)    -- (d2);
		
		
		\draw (f0)    -- (f1);
		\draw (f1)    -- (f2);
		
		\draw (e0)    -- (d1);
		\draw (e1)    -- (d2);
		
		\draw (e0)    -- (f1);
		\draw (e1)    -- (f2);
		
		\draw (d0)    -- (e1);
		\draw (d1)    -- (e2);
		
		\draw (f0)    -- (e1);
		\draw (f1)    -- (e2);
		
		
		\node[inner sep=0pt](p)    at ( 4,0) {$\bullet$};
		
		\node[inner sep=0pt](d0)  at (3,1) {$\bullet$};
		\node[inner sep=0pt](d1) at ( 3,2) {$\bullet$};
		\node[inner sep=0pt](d2)  at (3,3) {$\bullet$};
		\node[inner sep=0pt](d3)  at ( 3,4) {$\bullet$};
		
		\node[inner sep=0pt](e0)  at (4,1) {$\bullet$};
		\node[inner sep=0pt](e1) at ( 4,2) {$\bullet$};
		\node[inner sep=0pt](e2)  at (4,3) {$\bullet$};
		\node[inner sep=0pt](e3)  at (4,4) {$\bullet$};
		
		\node[inner sep=0pt](f0)  at (5,1) {$\bullet$};
		\node[inner sep=0pt](f1) at ( 5,2) {$\bullet$};
		\node[inner sep=0pt](f2)  at (5,3) {$\bullet$};
		\node[inner sep=0pt](f3)  at (5,4) {$\bullet$};
		
		\draw (p)    -- (d0);
		\draw (p)    -- (e0);
		\draw (p)    -- (f0);
		
		\draw (d0)    -- (d2);
		\draw (d1)    -- (d2);
		\draw (d2)    -- (d3);
		
		\draw (f0)    -- (f1);
		\draw (f1)    -- (f2);
		\draw (f2)    -- (f3);
		
		\draw (e0)    -- (d1);
		\draw (e1)    -- (d2);
		\draw (e2)    -- (d3);
		
		\draw (e0)    -- (f1);
		\draw (e1)    -- (f2);
		\draw (e2)    -- (f3);

		\draw (d0)    -- (e1);
		\draw (d1)    -- (e2);
		\draw (d2)    -- (e3);
		
		\draw (f0)    -- (e1);
		\draw (f1)    -- (e2);
		\draw (f2)    -- (e3);
		\end{tikzpicture}
		\caption{The Antichain $\Delta_0$}	
		\label{Delta0}
		
	\end{center}
\end{figure}
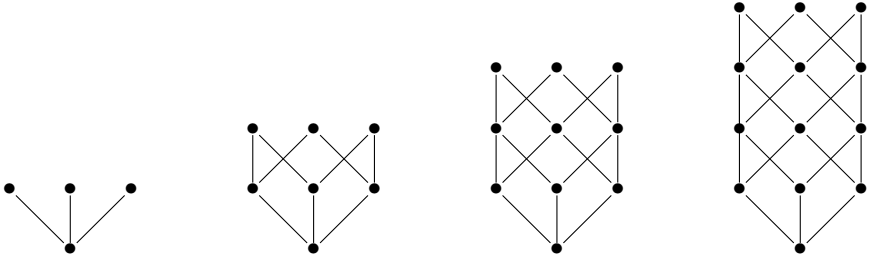

\noindent We let $\Delta_0:=\{ \fF_n \mid  n<\omega  \}$ be the set of all such posets. The following result follows by noticing that we cannot perform neither $\alpha$ nor $\beta$-reductions on any $\fF_n$ without collapsing the maximal elements---we refer the reader to \cite[Lem. 3.4.19]{bezhanishvili2006lattices} for a proof of this fact.

\begin{proposition}\label{antichain.1}
	The set of Heyting algebras dual to $\Delta_0$ is a $\leq$-antichain. 
\end{proposition}

\noindent Since every poset in $\Delta_0$ is finite and rooted, it follows immediately by Esakia duality that its algebraic duals are finite, subdirectly irreducible, Heyting algebras. In order to establish our result we  also need to make sure that every Heyting algebra which we are dealing with is regularly generated. This follows from the characterisation of finite regular posets which we provided in  \cref{sec.characterisation}. We recall that, if $\fF$ is a finite poset, then the \textit{depth} of an element $x\in \fF$, written $\depth(x)$, is defined as the size of the maximal chain in $x^\uparrow\setminus\{x\}$. Clearly the depth of a maximal elements is 0 and, in each $\fF_n$ the depth of the root is $n+1$.

\begin{proposition}\label{antichain.regularity}
	For every $n<\omega$, the poset $\fF_n$ is regular. In particular,  $\fF_n/\eqrel_n=\fF_n$ for every $n<\omega$.
\end{proposition}
\begin{proof}
	Consider $\fF_n$ for some $n<\omega$, we prove by induction on $\depth(x)$ that $[x]_k = \{x \}  $ whenever $k\geq \depth(x)-1$, for $\depth(x)>0$, and $k\geq 0$ otherwise.
	\begin{itemize}
		\item Let $\depth(x)\leq 1$. Without loss of generality we let $x=a_{1}$. Then, given any $y\in \fF_n$ such that $y\neq a_1$, we clearly have $M(a_1)\neq M(y)$, showing $[a_1]_0=\{a_1\}$.
		
		\item Let $\depth(x)=m+1< n+1$.  Without loss of generality we let $x=a_{m+1}$ and by induction hypothesis $[a_{l}]_{k}=\{a_{l}\}$, $[b_{l}]_{k}=\{b_{l}\}$ and $[c_{l}]_{k}=\{c_{l}\}$ whenever $l \leq m$, $k\geq l-1$. Now, for all $y\in \fF_n$ such that $x\neq y$, if $\depth(y)\leq m$ then $[y]_{m}=\{y\}$, thus  $[y]_{k}=\{y\}$  for all $k\geq m$. Otherwise, if $\depth(y)>m$ then since $y\neq a_{m+1}$, it follows $y\leq c_{m}$. Since $[c_m]_{m-1}=\{c_m\}$, this proves $x\neqrel_{m}y$ and thus by \cref{lemma_inductive} $x\neqrel_{k}y$ for all $k\geq m$. It follows $[x]_{k} =\{x\}$ for all $k\geq m=\depth(x)-1$.
		
		\item Let $\depth(x)=n+1$. The only point with depth $n+1$ in $\fF_n$ is the root $r$ and clearly it is the only point in $[r]_{n}$.
	\end{itemize}
	\noindent It follows that  $[x]_n=\{x\}$ for all $x\in \fF_n$ and thus $\fF_n/\eqrel_n =\fF_n$.
\end{proof}


Once we know that every poset from the antichain $\Delta_0$ above is regular, it is then straightforward to reason as in Jankov's original proof and show that the cardinality of the lattices of $\dna$-logics and $\dna$-varieties is exactly $2^{\aleph_0}$. We say that a finite Heyting algebra has \emph{width} (or \emph{depth}) $n$ if its dual poset has width (or depth) $n$,

\begin{theorem}\label{continuum.logics}
	There are continuum-many $\dna$-logics and $\dna$-varieties. In particular, there are continuum-many $\dna$-varieties generated by Heyting algebras of width 3.
\end{theorem}
\begin{proof}
	Let $\cA_0$ be the set of Heyting algebras dual to the posets in $\Delta_0$. Since every poset in $\Delta_0$ has width 3, the same holds for the dual Heyting algebras. By \cref{antichain.1}, $\cA_0$ is an infinite $\leq$-antichain of finite, subdirectly irreducible Heyting algebras. Moreover, by \cref{characterisation.regular} and \cref{antichain.regularity} we also have that each Heyting algebra in $\cA_0$ is regularly generated. By Theorem \ref{antichain.jankov} we have $Log^\neg(\mathcal{I})\neq Log^\neg(\mathcal{J})$ whenever $\mathcal{I},\mathcal{J}\subseteq \cA_0$ and $\mathcal{I}\neq \mathcal{J}$. By duality, we also have  $\dvari(\mathcal{I})\neq \dvari(\mathcal{J})$ whenever $\mathcal{I},\mathcal{J}\subseteq \cA_0$ and $\mathcal{I}\neq \mathcal{J}$, where  $\dvari(\class)$ denotes the $\dna$-variety generated by $\class$. Since $|\Delta_0|=\omega$, our result follows immediately.
\end{proof}

\noindent We also remark that, given the fact that $\dna$-varieties are in one-to-one correspondence with varieties of Heyting algebras generated by regular algebras, this also shows the existence of continuum-many varieties of Heyting algebras generated by regular elements. 

\subsection{Antichain $\Delta_1$}\label{antichain1}	Interestingly, we can also apply another standard example of infinite $\leq$-antichain to our context, originally due to Kuznetsov \cite{kuznetsov}, and show that there are continuum-many subvarieties of Heyting algebras which are generated by strongly regular elements.  

We recall the following construction and redirect the reader to \cite[\S 3]{bezhanishvili2022jankov} for more details. For every $1<n<\omega$, we let $\fG_n$ be the poset with domain
\[  \dom(\fG_n)=\{ r\} \cup \{a_m \mid m\leq n \} \cup \{b_m\mid m\leq n \}  \]
\noindent and such that
\begin{itemize}
	\item $r\leq a_i, b_i, \text{ for all } i\leq n;$
	\item $a_0\leq b_j,  \text{ for all } 0\leq j< n;$
	\item $a_n\leq b_j,  \text{ for all } 0<j\leq  n;$
	\item $a_i\leq b_j, \text{ for all } 0< i< n \text{ and } i\neq j.$
\end{itemize}

\noindent We let $\Delta_1:=\{ \fG_n\mid 1<n<\omega  \}$ be the set of all such posets. One can check that, whenever we collapse two maximal points in a frame $\fG_{n+1}$, the result is that every point of depth 1 is related to every point of depth 0, which is not the case in $\fG_n$. We thus obtain the following proposition, whose detailed proof is left to the reader.

\begin{proposition}\label{antichain.2}
	The set of Heyting algebras dual to $\Delta_1$ is a $\leq$-antichain. 
\end{proposition}

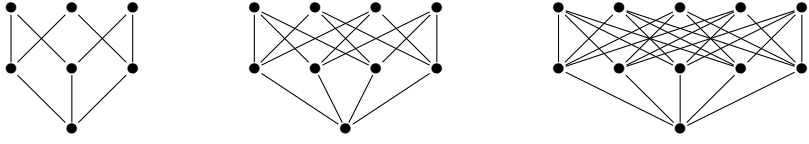
\begin{figure}
	\begin{center}
		\begin{tikzpicture}[scale=0.8]
		
		\node[inner sep=0pt](u)    at ( -5,0) {$\bullet$};
		
		\node[inner sep=0pt](x0)  at (-6,1) {$\bullet$};
		\node[inner sep=0pt](x1) at ( -6,2) {$\bullet$};
		
		\node[inner sep=0pt](y0)  at (-5,1) {$\bullet$};
		\node[inner sep=0pt](y1) at ( -5,2) {$\bullet$};
		
		\node[inner sep=0pt](z0)  at (-4,1) {$\bullet$};
		\node[inner sep=0pt](z1) at ( -4,2) {$\bullet$};
		
		\draw (u)    -- (x0);
		\draw (u)    -- (y0);
		\draw (u)    -- (z0);
		
		\draw (x0)    -- (x1);
		\draw (z0)    -- (z1);
		\draw (x0)    -- (y1);
		\draw (y0)    -- (z1);
		\draw (y0)    -- (x1);
		\draw (z0)    -- (y1);

		
		\node[inner sep=0pt](p)    at ( -0.5,0) {$\bullet$};
		
		\node[inner sep=0pt](d0)  at (-2,1) {$\bullet$};
		\node[inner sep=0pt](d1) at ( -1,1) {$\bullet$};
		\node[inner sep=0pt](d2)  at (0,1) {$\bullet$};
		\node[inner sep=0pt](d3)  at (1,1) {$\bullet$};
		
		\node[inner sep=0pt](e0)  at (-2,2) {$\bullet$};
		\node[inner sep=0pt](e1) at ( -1,2) {$\bullet$};
		\node[inner sep=0pt](e2)  at (0,2) {$\bullet$};
		\node[inner sep=0pt](e3)  at (1,2) {$\bullet$};
		
		\draw (p)    -- (d0);
		\draw (p)    -- (d1);
		\draw (p)    -- (d2);
		\draw (p)    -- (d3);
		
		\draw (d0)    -- (e0);
		\draw (d0)    -- (e1);
		\draw (d0)    -- (e2);
		
		\draw (d3)    -- (e3);
		\draw (d3)    -- (e1);
		\draw (d3)    -- (e2);
		
		\draw (d2)    -- (e3);
		\draw (d2)    -- (e1);
		\draw (d2)    -- (e0);
		
		\draw (d1)    -- (e3);
		\draw (d1)    -- (e2);
		\draw (d1)    -- (e0);
		
		\node[inner sep=0pt](m)    at ( 5,0) {$\bullet$};
		
		\node[inner sep=0pt](m0)  at (3,1) {$\bullet$};
		\node[inner sep=0pt](m1) at ( 4,1) {$\bullet$};
		\node[inner sep=0pt](m2)  at (5,1) {$\bullet$};
		\node[inner sep=0pt](m3)  at (6,1) {$\bullet$};
		\node[inner sep=0pt](m4)  at (7,1) {$\bullet$};
		
		\node[inner sep=0pt](n0)  at (3,2) {$\bullet$};
		\node[inner sep=0pt](n1) at ( 4,2) {$\bullet$};
		\node[inner sep=0pt](n2)  at (5,2) {$\bullet$};
		\node[inner sep=0pt](n3)  at (6,2) {$\bullet$};
		\node[inner sep=0pt](n4)  at (7,2) {$\bullet$};
		
		\draw (m)    -- (m0);
		\draw (m)    -- (m1);
		\draw (m)    -- (m2);
		\draw (m)    -- (m3);
		\draw (m)    -- (m4);
		
		\draw (m0)    -- (n0);
		\draw (m0)    -- (n1);
		\draw (m0)    -- (n2);
		\draw (m0)    -- (n3);
		
		\draw (m4)    -- (n4);
		\draw (m4)    -- (n1);
		\draw (m4)    -- (n2);
		\draw (m4)    -- (n3);
		
		\draw (m3)    -- (n0);
		\draw (m3)    -- (n1);
		\draw (m3)    -- (n2);
		\draw (m3)    -- (n4);

		\draw (m2)    -- (n3);
		\draw (m2)    -- (n1);
		\draw (m2)    -- (n0);
		\draw (m2)    -- (n4);
		
		\draw (m1)    -- (n3);
		\draw (m1)    -- (n2);
		\draw (m1)    -- (n0);
		\draw (m1)    -- (n4);
		\end{tikzpicture}
		\caption{The Antichain $\Delta_1$}	
		\label{Delta1}
	\end{center}
\end{figure}

\noindent As the posets in $\Delta_1$ grow in width rather than depth, we have that every $\fG_n$ is stable already under the quotient $\eqrel_0$, i.e. $\fG_n/\eqrel_0=\fG_n$ for all $1<n<\omega$. So every poset in $\Delta_1$ is actually strongly regular.

\begin{proposition}\label{antichain.regularity.2}
	Every poset $\fG_n$ is strongly regular.
\end{proposition}
\begin{proof}
	By construction, it is straightforward to check that any two different $x,y\in \fG_n$ see different maximal elements, i.e. $M(x)\neq M(y)$.
\end{proof}

\noindent By the same reasoning as above, we immediately obtain a second uncountable family of $\dna$-varieties and $\dna$-logics.

\begin{theorem}
	There are continuum many $\dna$-varieties generated by strongly regular Heyting algebras of depth 3.
\end{theorem}
\begin{proof}
	Analogously to \cref{continuum.logics}, together with the fact that posets from $\Delta_1$ are strongly regular and have depth 3.
\end{proof}

\noindent Also, this means that there are continuum-many varieties of Heyting algebras generated by strongly regular elements of width 3.

\section{Applications to Logic}\label{applications.logic}

In this section we consider some applications to logic of regular Heyting algebras. As we saw in \cref{subsection:dnaLogicsAndAlgebraicSemantics}, regular Heyting algebras play an important role in the algebraic semantics of $\dna$-logics. We employ Esakia duality to adapt these results to the topological setting, thereby obtaining a topological semantics for $\dna$-logics. Secondly, we consider the case of dependence logic and we extend this topological semantics to this setting as well. We start by adapting the notion of $\dna$-variety to the context of Esakia spaces.

\subsection{$\dna$-Varieties of Esakia Spaces}

In analogy with the algebraic case, we define a special family of varieties of Esakia spaces, closed under an additional operation which preserves the structure of the regular clopen upsets.

\begin{definition}[$\dna$-variety of Esakia spaces]
	A \emph{$\dna$-variety of Esakia spaces} $\mathcal{E}$ is a variety of Esakia spaces additionally closed under the following operation:
	\begin{equation*}
	\mathcal{E}^{M} = \{ \espace{E} \mid  \exists \espace{F}\in \mathcal{E}.\; \exists f: \espace{E} \twoheadrightarrow \espace{F}.\; f\restriction M_{\espace{E}} \text{ is a homeomorphism of Stone spaces}  \}.
	\end{equation*}
\end{definition}


\noindent  Given a class $\cE$ of Esakia spaces, we write $\mathbb{S}(\cE)$ for the smallest $\dna$-variety of Esakia spaces containing $\cE$ and we denote by $\Lambda(\Esa^{M})$ the sublattice of $\Lambda(\Esa)$ consisting of $\dna$-varieties. When we restrict Esakia duality to $\dna$-varieties of Heyting algebras and $\dna$-varieties of Esakia spaces, we immediately obtain the following theorem.
\begin{theorem}\label{correspondence}
	The maps $\overline{\PF}$ and $\overline{\clop\upset}$ restricted to the sublattices $\Lambda(\HA^{\uparrow})$ and $\Lambda(\Esa^{M})$ induce an isomorphism of the two lattices.
\end{theorem}
\begin{proof}
	Notice that, given $H$ and $K$ Heyting algebras, the two conditions
	\begin{itemize}
		\item[(i)] $ K_{\neg} = H_{\neg} \text{ and } K \preceq H $;
		\item[(ii)]$ \exists f: \fE_H \twoheadrightarrow \fE_K.\; f^{-1}: \mathcal{RCU}(\fE_K) \to \mathcal{RCU}(\fE_H)$ is a Boolean algebras isomorphism;
	\end{itemize}
	\noindent are dual to each other. By \cref{maximalinjectivity} we have that (ii) is equivalent to the following claim
	\begin{itemize}
		\item[(iii)] $\exists f: \fE_H \twoheadrightarrow \fE_K.\; f\restriction M_{\espace{E}} \text{ is a homeomorphism of Stone spaces}$.
	\end{itemize}
	\noindent  Given this, we have that $\dna$-varieties of Heyting algebras are in one-to-one correspondence to $\dna$-varieties of Esakia spaces, from which it is immediate to verify the main statement.	    
\end{proof}

Finally, we notice that in \cite{bezhanishvili_grilletti_quadrellaro_2022} we proved several results concerning $\dna$-varieties of Heyting algebras, which is straightforward to adapt to $\dna$-varieties of Esakia spaces. In particular, we recall the following Birkhoff's style theorem. We say that a class of Heyting algebras $\class$ has the \textit{$\dna$-finite model property} if whenever $\class\nvDash^\neg \phi$ there is  some finite $H\in \class$ such that  $H\nvDash^\neg  \phi$.
\begin{theorem}\label{birkhoff} $\;$
	\begin{enumerate}
		\item Every $\dna$-variety of Heyting algebras $\cX$ is generated by its collection of regular, subdirectly irreducible elements, i.e. $\cX = \dvari(\cX_{RSI}) $.			
		\item If a $\dna$-variety $\cX$ has the $\dna$-finite model property, it is generated by its finite, regular, subdirectly irreducible elements, i.e. $\cX=\dvari(\cX_{RFSI}) $.
	\end{enumerate}		
\end{theorem}

\noindent Using Esakia duality it is immediate to translate this result to $\dna$-varieties of Esakia spaces.  We recall that a Heyting algebra is subdirectly irreducible if and only its dual Esakia space is \textit{strongly rooted}, i.e. if it has a least element $r$ such that $\{ r \}$ is open (see \cite[p. 152]{esakia1979theory} and \cite[Thm. 2.9]{bezhanishvili2008profinite}). A $\dna$-variety of Esakia spaces has the $\dna$-finite model property if its dual $\dna$-variety of Heyting algebras has this property.

\begin{corollary}\label{birkhoff2} $\;$
	\begin{enumerate}
		\item Every $\dna$-variety of Esakia spaces $\cE$ is generated by its collection of regular, strongly rooted elements, i.e. $\cE = \mathbb{S}(\cE_{RSI}) $.			
		\item If a $\dna$-variety $\cE$ has the $\dna$-finite model property, then it is generated by its rooted, finite, regular elements, i.e. $\cE=\mathbb{S}(\cE_{RFR}) $.
	\end{enumerate}		
\end{corollary} 


\subsection{$\dna$-Logics and Inquisitive Logic}

We introduce a topological semantics for $\dna$-logics that mirrors their algebraic semantics. The results of \cref{subsection:stoneSpaceMaximalElements} suggest to define a semantics for $\dna$-logics in terms of Esakia spaces and regular clopen upsets. 

Given an Esakia space $\espace{E}$ we call a  function $\mu:\at\rightarrow \mathcal{RCU}(\espace{E})$ a \textit{$\dna$-valuation} over $\espace{E}$.
For $\mu$ a $\dna$-valuation, define the interpretation of formulas over $\espace{E}$ as follows:
\begin{equation*}
\begin{array}{r@{\hspace{.1em}}l @{\hspace{1em}}  r@{\hspace{.1em}}l @{\hspace{1em}}  r@{\hspace{.1em}}c@{\hspace{.1em}}l}	
\llbracket p \rrbracket^\mathfrak{\espace{E},\mu} &= \mu(p)
&\llbracket \bot \rrbracket^\mathfrak{\espace{E},\mu} &= \emptyset \\[0.5em]
\llbracket \top \rrbracket^\mathfrak{\espace{E},\mu} &=  \mathfrak{E}  
&\llbracket \phi \land \psi \rrbracket^\mathfrak{\espace{E},\mu} &= \llbracket \phi \rrbracket^{\espace{E},\mu} \cap \llbracket \psi \rrbracket^\mathfrak{\espace{E},\mu}  \\[0.5em]
\llbracket \phi \rightarrow \psi \rrbracket^\mathfrak{\espace{E},\mu} &= \overline{ \llbracket \phi \rrbracket^\mathfrak{\espace{E},\mu} \setminus \llbracket \psi \rrbracket^{\mathfrak{\espace{E},\mu}} }
&\llbracket \phi \lor \psi \rrbracket^\mathfrak{\espace{E},\mu} &= \llbracket\phi \rrbracket^{\mathfrak{\espace{E},\mu}} \cup  \llbracket\psi \rrbracket^{\mathfrak{\espace{E},\mu}}.
\end{array}
\end{equation*}

\smallskip

\noindent
The only difference with the definition of \cref{subsection:semanticsIntermediate} being that in the atomic case the interpretation is restricted to $\mathcal{RCU}(\espace{E})$.
Notice however that not \emph{all} formulas have to range over the set $\mathcal{RCU}(\espace{E})$. For example, it is not true in general that the union of two regular sets is regular, and in fact $\sem{p \vee q}^{\espace{E},\mu}$ may be a non-regular element of $\clop\upset(\espace{E})$. We then say that a formula $\phi$ is \emph{$\dna$-valid} on a space $\espace{E}$ ($\espace{E} \vDash^{\neg} \phi$) if $\sem{\phi}^{\espace{E},\mu} = \espace{E}$ for every $\dna$-valuation $\mu$. We say that a formula $\phi$ is $\dna$-valid on a class of spaces $\cE$ ($\cE \vDash^{\neg} \phi$) if it is $\dna$-valid on every element of the class. We write $Log^\neg(\cE)$ for the set of the $\dna$-valid formulas of $\cE$ and we write $Space^{\neg}(\mathtt{L}) $ for the $\dna$-variety of Esakia spaces which validate all formulas in $\mathtt{L}$.

Since $\dna$-valuations over Esakia spaces correspond through Esakia duality exactly to negative valuations over their dual Heyting algebras, the algebraic completeness of $\dna$-logics immediately  establishes the completeness of this topological semantics. 

\begin{theorem}\label{algebraic.completeness.dna}	
	Let $\mathtt{L}$ be a $\dna$-logic,  $\mathcal{E}$ a $\dna$-variety of Esakia Spaces, $\phi$ a formula and  $\mathfrak{E}$ an Esakia Space. Then we have the following:
	\begin{align*}
	\phi \in \mathtt{L} &\Longleftrightarrow Space^{\neg}(\mathtt{L})\vDash^\neg \phi; \\
	\mathfrak{E}\in \mathcal{E} & \Longleftrightarrow \mathfrak{E}\vDash^\neg Log^{\neg}(\mathcal{E}).
	\end{align*}
\end{theorem}

We remark that this also delivers a topological semantics for inquisitive logic which differs from the one previously studied in \cite{grilletti}, rather based on UV-spaces. Since inquisitive logic $\inqB$ is the negative variant of any intermediate logic between $\mathtt{ND}$ and $\mathtt{ML}$, the former result shows that inquisitive logic also admits a topological semantics based on Esakia spaces, which mirrors its algebraic semantics based on regular Heyting algebras. 
\begin{corollary}\label{inquisitive.completeness}
	Let $L$ be any intermediate logic between $\mathtt{ND}$ and $\mathtt{ML}$, then $\phi\in \inqB $ if and only if $  Space^{\neg}(L)\vDash^\neg \phi$.
\end{corollary}

\subsection{Dependence Logic}\label{dependence.logic}

We conclude by showing how the previous topological semantics can be extended to dependence logic, which, in its propositional version, can be seen as an extension of inquisitive logic in a larger signature.

Originally, dependence logic was introduced by V\"a\"an\"anen \cite{Vaananen2007-VNNDLA} as an extension of first-order logic with dependence atoms. A key aspect of dependence logic is that it is formulated in so-called \textit{team-semantics}, which was introduced by Hodges in \cite{Hodges}. In its propositional version, which was developed by Yang and V\"a\"an\"anen in \cite{Yang2016-YANPLO,Yang2017-YANPTL}, teams are simply sets of propositional assignments. It was soon observed in Yang's thesis \cite{yang2014extensions} -- see also \cite{Ciardelli2016,Yang2016-YANPLO} -- that the team semantics of propositional dependence logic actually coincides with the state-based  semantics of inquisitive logic, thus establishing an important connection between dependence and inquisitive logic.

We explore here a further aspect of this connection and we illustrate the relation between propositional dependence logic and regular Esakia spaces. In particular, we will adapt the completeness proof of \cref{algebraic.completeness.dna} so as to obtain a sound and complete topological semantics for dependence logic.

\subsubsection{Syntax and Semantics}

Propositional dependence logic can be seen as an extension of inquisitive logic in a larger vocabulary $\langInqI$, which adds the so-called tensor operator to the signature of intuitionistic logic. Formulas of dependence logic are thus defined recursively as follows:
\[
\phi ::= p \mid  \bot   \mid  \phi \land \phi \mid \phi\otimes \phi \mid \phi \lor \phi \mid \phi\rightarrow\phi,\]

\noindent where $p\in\at$. We define $\neg\alpha:=\alpha\to\bot$ and we say that a formula is \textit{standard} if it does not contain any instance of $\lor$. We provide this syntax with the usual team semantics. We recall that a propositional \textit{assignment} is a map $w:\at\to 2$ and that a \textit{team} is a set of assignments $t\in \wp( 2^\at)$. The team semantics of dependence logic is then defined as follows. 

\begin{definition}[Team Semantics]
	The notion of a formula $\phi\in\langInqI$ being \textit{true in a team} $t\in \wp({2^\at})$ is defined as follows: 	
	\begin{equation*}
	\begin{array}{l @{\hspace{1em}\Longleftrightarrow\hspace{1em}} l}
	t\vDash p & {}\forall w\in t \ ( w(p)=1)  \\
	t\vDash \bot &  t=\emptyset \\
	t\vDash \psi \lor \chi & t\vDash \psi \text{ or } t\vDash \chi\\
	t\vDash \psi \land \chi & t\vDash \psi \text{ and } t\vDash \chi\\
	t\vDash \psi \otimes \chi & \exists s,r\subseteq t \text{ such that } s\cup r = t \text{ and } s\vDash \psi, r\vDash \chi \\
	t\vDash \psi \rightarrow \chi &  \forall s \ ( \text{if }s\subseteq t \text{ and } s\vDash\psi \text{ then } s\vDash \chi ).
	\end{array}
	\end{equation*}
\end{definition}

\noindent  We define \emph{propositional dependence logic} as the set  $\inqB^\otimes= Log( \wp(2^\at))$ of all formulas of $\langInqI$ valid under team semantics. We notice that inquisitive logic has exactly the same semantics but it is formulated in the restricted language $\langInt$, which lacks the tensor disjunction $\otimes$, thus in particular $\inqB^\otimes\supseteq \inqB$. The following normal form was proven in \cite{Ciardelli.2009} for inquisitive logic and extended in \cite{Yang2016-YANPLO} to dependence logic.

\begin{theorem}[Disjunctive Normal Form] \label{disjunctive.normal.form}
	Let $\phi \in \langInqI$, then there are standard formulas $\alpha_0,\dots,\alpha_n\in \langInqI$ such that $\phi \equiv_{\inqB^\otimes} \bigvee_{i\leq n} \alpha _i  $.
\end{theorem}

We finally remark that the propositional dependence atom can be defined in this system as follows:
\[ \dep (p_0,\dots, p_n, q):= \bigwedge_{i\leq n} (p_i\lor \neg p_i ) \to (q\lor \neg q). \]
\noindent We thus notice that, despite the name, it is not the dependence atom which distinguishes the propositional version of inquisitive and dependence logics, but rather the presence of the tensor. This observation is also justified by the work of Barbero and Ciardelli in  \cite{ciardelli2019undefinability}, as they showed that the tensor cannot be uniformly defined by the other operators. 

\subsubsection{Algebraic Semantics of Dependence Logic} As we have recalled above, inquisitive logic admits a (non-standard) algebraic semantics, which was introduced in \cite{grilletti} and further investigated in \cite{bezhanishvili_grilletti_quadrellaro_2022}. As dependence logic extends inquisitive logic by the tensor operator, it is natural to provide it with an algebraic semantics by augmenting inquisitive algebras with an interpretation for it. Such a semantics was first introduced in \cite{quadrellaro2021intermediate} and was later shown in  \cite{nakov.quadrellaro} to be unique up to a suitable notion of algebraizability. We can use such algebraic semantics to build a bridge with Esakia spaces and provide a topological semantics for dependence logic. Firstly, we introduce the notion of $\inqB^\otimes$-algebras as in \cite{nakov.quadrellaro}.

\begin{definition}
	An \emph{$\inqB^\otimes$-algebra} $A$ is a structure in the signature $\langInqI$ such that: 
	\begin{enumerate}
		\item $A{ \restriction } \{\lor,\land,\to, \bot  \} \in Var(\mathtt{ML})  $;
		\item $A_\neg{ \restriction } \{\otimes,\land,\to, \bot  \} \in \BA$;
		\item $A \vDash x \otimes (y \lor z) \approx (x\otimes y) \lor (x\otimes z);$
		\item $A \vDash (x\to z) \to (y \to k) \approx  (x\otimes y) \to (z\otimes k).$
	\end{enumerate}
\end{definition}

\noindent Hence, an $\inqB^\otimes$-algebra is the expansion of  a Heyting algebra satisfying the validities of $\mathtt{ML}$, and the additional conditions above.  By expanding the previous definition, one can see that it amounts to the equational definition of a class of algebras, thus giving rise to a variety of structures. Notice that, as the regular elements of a Heyting algebra always form a Boolean algebra, what the condition $\cA_\neg{ \restriction } \{\otimes,\land,\to, \bot  \} \in \BA$ really entails is that, for all regular elements $x,y\in \cA_\neg$, $x\otimes y := \neg(\neg x \land \neg y)$, i.e. the tensor is the “real” Boolean disjunction over regular elements.

We let $ \mathsf{InqBAlg^\otimes} $ be the variety of all $\inqB^\otimes$-algebras and  we write $ \mathsf{InqBAlg^\otimes_{FRSI}} $ for its subclass of finite, regular and subdirectly irreducible elements. We say that $A$ is a \emph{dependence algebra} if it belongs to the subvariety generated by all finite, regular, subdirectly irreducible $\inqB$-algebras, i.e. if $A\in \mathbb{V}(\mathsf{InqBAlg^\otimes_{FRSI}})$. We write $\mathsf{DA}:= \mathbb{V}(\mathsf{InqBAlg^\otimes_{FRSI}})$ for the variety of dependence algebras. It was proven in  \cite{nakov.quadrellaro} that $\mathsf{DA}$ is the equivalent algebraic semantics of $\inqB^\otimes$. In particular, we have the following completeness result:

\begin{theorem}[Algebraic Completeness]\label{algebraic.complete.dependence}
	For any formula $\phi\in \langInqI$ we have that  $\phi \in \inqB^\otimes$ if and only if $ \mathsf{DA}  \vDash^\neg \phi$.
\end{theorem}

\noindent Where on the right hand side we are using the same notion of truth of \cref{subsection:dnaLogicsAndAlgebraicSemantics}, i.e. formulas of dependence logic are evaluated under negative valuations, which map atomic formulas to regular elements of the underlying dependence algebra.

\subsubsection{Topological Semantics of Dependence Logic} 
The algebraic semantics of propositional dependence logic makes for an important bridge with the topological approach that we developed in this article. In fact, dependence algebra are expansions of Heyting algebras (more specifically of $\mathtt{ML}$-algebras), whence we can dualize them according to Esakia duality. The only problem when proceeding in this way is that, as the Esakia duality  accounts only for the Heyting algebra structure of a dependence algebra, the correct interpretation of the tensor operator is “lost in translation”. To avoid this problem we shall consider only regular dependence algebras.

Let $\fE$ be a regular  Esakia space satisfying $\mathtt{ML}$, it is easy to provide an interpretation for the tensor over clopen upsets $\fE$. In fact, as we remarked previously, the tensor of two regular elements is simply their classical Boolean disjunction. Moreover, it follows immediately from the disjunctive normal form of $\inqB$ (\cref{disjunctive.normal.form}) and the fact that $\mathtt{ML}$-spaces are complete with respect to $\inqB$ (\cref{inquisitive.completeness}) that any clopen upset of a regular $\mathtt{ML}$-Esakia space is a union of regular ones. This allows us to define the tensor operator over $\mathcal{CU}(\fE)$ as follows:
\begin{enumerate}
	\item For  $U,V\in \mathcal{RCU}(\fE)$ we let $U\otimes V := \overline{(\overline{U} \cup \overline{V})}$;
	\item For  $U,V\in \mathcal{CU}(\fE)\setminus \mathcal{RCU}(\fE)$ we let \[U\otimes V := \bigcup \{ U_0 \otimes V_0\mid  U_0\subseteq U, V_0\subseteq V, U_0,V_0\in  \mathcal{RCU}(\fE) \}.\]
\end{enumerate}

\noindent We leave it to the reader to verify that $\mathcal{CU}(\fE)$ forms a dependence algebra, where the tensor operator is interpreted as we remarked. However, although this definition suffices in explaining how the tensor can be interpreted over algebras of clopen upsets, it still does not provide us with a topological intuition of its behaviour. To this end, we prove the following proposition. 

\begin{proposition}\label{tensor.description}
	Let $\fE$ be a regular Esakia space satisfying $\mathtt{ML}$, and let $\otimes$ be defined by the clauses above, then we have, for any $U,V\in \CU(\fE)$: 
	\begin{align*}
	x\in U\otimes V \Longleftrightarrow  \; & M(x)\subseteq U_0\cup V_0 \\\; &\text{for some } U_0\subseteq U \text{ and } V_0\subseteq V \text{ such that } U_0,V_0\in \mathcal{C}(M_\fE).
	\end{align*}
\end{proposition}
\begin{proof}
	Firstly, if  $U,V\in \mathcal{RCU}(\fE)$ we have $U\otimes V = \overline{(\overline{U} \cup \overline{V})}$. We obtain:
	\begin{align*}
	x\in \overline{(\overline{U} \cup \overline{V})} &\Longleftrightarrow  \forall y\geq x, \; y\notin \overline{U} \cap \overline{V} \\
	&\Longleftrightarrow  \forall y\geq x \; \exists z \geq y, \; z\in U \cup V \\
	&\Longleftrightarrow M(x)\subseteq U \cup V\\
	&\Longleftrightarrow M(x)\subseteq M(U) \cup M(V).
	\end{align*}
	\noindent Then, for arbitrary  $U,V\in \mathcal{CU}(\fE)$, the claim follows immediately from the definition of the tensor and the display above.
\end{proof}

\noindent The previous proposition thus provides us with a topological interpretation for the tensor operator and shows that the tensor disjunction between two clopen upsets of an Esakia space is uniquely determined by the Stone subspace of its maximal elements. 

Now, let $\mathsf{Esa^{\mathtt{ML}}_{RFR}}$ be the class of rooted, finite and regular posets which satisfy $\mathtt{ML}$ and augment them by a tensor operator defined as in \cref{tensor.description}. By the definition of the variety of dependence algebras it follows that the validity of $\inqB^\otimes$-formulas is always witnessed by finite, regular, subdirectly irreducible algebras (see also \cite{quadrellaro2021intermediate}). The following theorem thus follows exactly as \cref{algebraic.completeness.dna}, by applying Esakia duality and  interpreting the tensor as we illustrated above.

\begin{theorem}[Topological Completeness]\label{topological.complete.dependence}
	For any formula $\phi\in \langInqI$ we have that  $\phi \in \inqB^\otimes$ if and only if $ \mathsf{Esa^{\mathtt{ML}}_{RFR}} \vDash^\neg \phi$.
\end{theorem}

\noindent As the validity of formulas is preserved by the variety operations, we can extend the previous result and infer the completeness of $\inqB^\otimes$ with respect to the closure of the class $\mathsf{Esa^{\mathtt{ML}}_{RFR}}$ under subspaces, p-morphisms and coproducts. Notice, however, that our topological characterisation of the tensor operator is limited to regular Esakia spaces. The questions whether the tensor admits an interesting topological interpretation also in non-regular spaces should be subject of further investigations.

\section{Conclusion}\label{conclusion}

In this article we considered regular Heyting algebras from the point of view of Esakia duality and we  provided several results about their dual topological spaces. In particular, in  \cref{sec.characterisation} we described two different characterisations of (finite) regular Esakia spaces and in  \cref{cardinality.sublattice} we applied them to show that there are continuum many varieties of Heyting algebras generated by (strongly) regular elements. This also shows that there are continuum many $\dna$-varieties and $\dna$-logics, in contrast to the fact that there are only countably many extensions of inquisitive logic. Finally, in \cref{applications.logic},  we considered several logical applications of our work and we introduced novel topological semantics for $\dna$-logics, inquisitive logic and dependence logic, which crucially rely on regular Esakia spaces.

We believe that the present work hints at some possible directions of further research. Besides the questions already raised in the article, we wish here to bring three points to attention. 

Firstly, in \cite{quadrellaro2021intermediate} we have considered the algebraic semantics of a wide range of intermediate versions of inquisitive and dependence logics. As this semantics relies on Heyting algebras with a core of join-irreducible elements, it is then natural to ask to what extent one could extend the duality results of this article to this context.

Secondly, is it possible to extend our characterisation of finite regular posets from \cref{quotient.characterisation} to account also for infinite Esakia spaces? As we have briefly remarked, the cases of image-finite Esakia spaces, or of Esakia spaces  dual to finitely generated Heyting algebras do not pose serious problems, but in general this seems a non-trivial problem. 

Finally, the class of finite regular posets has a quite combinatorial nature and makes for an interesting class of structures. Is it possible to provide a classification of these structures up to some suitable notion of dimension, e.g. their depth or their number of maximal elements? We leave these and other problems to future research.

	\printbibliography

\end{document}